\documentclass[10pt,a4paper]{article}
\usepackage[a4paper,top=3cm,bottom=3cm,left=3cm,right=3cm,%
bindingoffset=5mm]{geometry}
\usepackage{ifpdf}
\pdfoutput=1
\usepackage[utf8]{inputenc}
\usepackage[english]{babel}
\usepackage{amsmath}
\usepackage{amsfonts}
\usepackage{amssymb}
\usepackage{graphicx}
\usepackage{amsthm}
\usepackage{epstopdf}
\usepackage{tabularx}
\usepackage{multirow}
\usepackage{booktabs}
\usepackage{color}
\usepackage{authblk}
\usepackage{caption}
\usepackage{subcaption}
\usepackage[percent]{overpic}
\newcommand{\numberset}{\mathbb}

\addcontentsline{toc}{chapter}{Chapter1}
\newtheorem{theorem}{Theorem}[section]

\newtheorem{remark}{Remark}[section]
\newtheorem{lemma}[theorem]{Lemma}
\newtheorem{proposition}{Proposition}[section]

\newcommand{\N}{\numberset{N}}
\newcommand{\R}{\numberset{R}}

\newcommand{\M}{\numberset{M}}

\newcommand{\Pk}{\numberset{P}}

\newcommand{\Pkw}{\widehat{\numberset{P}}}

\renewcommand{\epsilon}{\varepsilon}
\renewcommand{\theta}{\vartheta}
\renewcommand{\rho}{\varrho}
\renewcommand{\phi}{\varphi}

\newcommand{\xx}{\boldsymbol{x}}
\newcommand{\xxp}{\xx^{\perp}}
\newcommand{\ee}{\boldsymbol{e}}
\newcommand{\uu}{\boldsymbol{u}}
\newcommand{\vb}{\boldsymbol{v}}

\newcommand{\ww}{\boldsymbol{w}}
\newcommand{\ff}{\boldsymbol{f}}
\newcommand{\nn}{\boldsymbol{n}}

\newcommand{\ddd}{\boldsymbol{\delta}}
\newcommand{\dd}{{\rm div}}
\newcommand{\gr}{\nabla}
\newcommand{\Gr}{\boldsymbol{\nabla}}

\newcommand{\dl}{\boldsymbol{\Delta}}
\newcommand{\VV}{\boldsymbol{V}}
\newcommand{\W}{\Sigma}
\newcommand{\Wd}{\Sigma_{\rm D}}
\newcommand{\Wz}{\Sigma_0}

\newcommand{\ZZ}{\boldsymbol{Z}}

\def\P0{{\Pi^{0, E}_k}}
\def\Pp0{{\Pi^{0, E}_{k-1}}}

\def\PP0{{\boldsymbol{\Pi}^{0, E}_{k-1}}}

\newcommand{\epseps}{\boldsymbol{\epsilon}}

\def\av{{a_V}}

\def\avh{{a_{V,h}}}
\def\avhl{{a_{V,h}^E}}
\def\cv{{c_V}}
\def\cvh{{c_{V,h}^{\rm skew}}}
\def\cvhl{{c_{V,h}^{{\rm skew}, E}}}
\def\as{{a_T}}

\def\ash{{a_{T,h}}}
\def\ashl{{a_{T,h}^E}}
\def\cs{{c_T}}
\def\csh{{c_{T,h}^{\rm skew}}}
\def\cshl{{c_{T,h}^{{\rm skew}, E}}}
\newcommand{\cvskew}{c_V^{\rm skew}}
\newcommand{\csskew}{c_T^{\rm skew}}
\def\svl{{s_{V}^E}}
\def\ssl{{s_{T}^E}}

\def\reg{s}

\def\isd{\widehat{\beta}}

\newcommand{\vbi}{\vb_{\cal I}}
\newcommand{\uui}{\uu_{\cal I}}

\newcommand{\eei}{\ee_{\cal I}}
\newcommand{\si}{\sigma_{\cal I}}

\newcommand{\WDL}{\W_h(E)}
\newcommand{\WDG}{\W_h}
\newcommand{\WDZG}{\W_{0,h}}
\newcommand{\WDDG}{\W_{{\rm D}, h}}
\newcommand{\VCG}{\VV}
\newcommand{\VDL}{\VV_h(E)}
\newcommand{\VDG}{\VV_h}
\newcommand{\ZCG}{\ZZ}

\newcommand{\ZDG}{\ZZ_h}
\newcommand{\QCG}{Q}

\newcommand{\QDG}{Q_h}

\newcommand{\td}{\theta_{\rm D}}
\newcommand{\tdh}{\theta_{\rm D, h}}

\newcommand{\tinth}{\theta^o_h}
\newcommand{\tparh}{\theta^{\partial}_h}

\newcommand{\clip}{\nu_{\rm lip}}
\newcommand{\cclip}{\widehat{C}_{\rm lip}}
\newcommand{\ccon}{C_{\rm conv}}

\newcommand{\cest}{{C}_{\rm est}}
\newcommand{\cesth}{\widehat{C}_{\rm est}}
\newcommand{\cdata}{{C}_{\rm data}}
\newcommand{\cdatah}{\widehat{C}_{\rm data}}
\newcommand{\csol}{{C}_{\rm sol}}
\newcommand{\csolh}{\widehat{C}_{\rm sol}}

\newcommand{\cexi}{C_{\rm exi}}
\newcommand{\cconh}{\widehat{C}_{\rm conv}}

\newcommand{\avs}{\alpha_*}
\newcommand{\avS}{\alpha^*}
\newcommand{\ats}{\beta_*}
\newcommand{\atS}{\beta^*}

\newcommand{\cfun}[1]{\textnormal{\texttt{\detokenize{#1}}}}
\newcommand{\cfan}[1]{\textnormal{\textbf{\detokenize{#1}}}}

\def\keywords{\vspace{.5em}
{\textit{Keywords}:\,\relax%
}}

\author[1]{P. F. Antonietti \thanks{paola.antonietti@polimi.it}}

\author[2]{G. Vacca \thanks{giuseppe.vacca@uniba.it}}

\author[1]{M. Verani \thanks{marco.verani@polimi.it}}

\affil[1]{MOX-Laboratory for Modeling and Scientific Computing,  Dipartimento di Matematica, Politecnico di Milano, 
Piazza Leonardo da Vinci 32 - 20133 Milano, Italy}

\affil[2]{Dipartimento di Matematica, 
Universit\`a degli Studi di Bari, 
Via Edoardo Orabona 4  - 70125 Bari, Italy}

\begin{document}

\title{Virtual Element Method for the Navier--Stokes Equation coupled with 
the Heat Equation}
\maketitle
\begin{abstract}
We consider the Virtual Element discretization of the Navier–Stokes equations coupled with the heat equation where the viscosity depends on the temperature. We present the  Virtual Element discretization of the coupled problem, show its well-posedness, and prove optimal error estimates. Numerical experiments which confirm the theoretical error bounds are also presented.
\end{abstract}

\keywords{Navier–Stokes equations, heat equation, Virtual Element methods.}

\section{Introduction}

In this paper, we consider the stationary flow of a viscous incompressible fluid, in the case where the viscosity of the fluid depends on the temperature.  Such a fluid-thermal coupling appears in many physical phenomena and it is involved in wide range industrial applications, e.g.  cooling processes in steel industries, industrial furnaces, boilers, heat exchangers and nuclear power plants.

Due to its importance in many practical applications, the numerical approximation of coupled {N}avier-{S}tokes and heat equations have been widely  studied. In \cite{oyarzua:2014} a generalized Boussinesq problem modelling the motion of a nonisothermal incompressible fluid subject to a heat source has been discretized by means of mixed finite element methods. A spectral discretization of the {N}avier-{S}tokes equation coupled with the heat equation has been proposed in \cite{Bernardi:2015}, in the stationary case, and in \cite{Bernardi:2015-b} in the unsteady one. Finite Element approximation of the time dependent Boussinesq model with nonlinear viscosity depending on the temperature has been studied in \cite{Aldbaissy:2018}. Recently, the finite element approximation of the heat equation coupled with Stokes equations with nonlinear slip boundary conditions has been analyzed in \cite{Djoko:2020}. Finite element methods for {D}arcy's problem coupled with the heat equation has been studied in \cite{Bernardi:2018}.

In this paper we are interested in the Virtual Element discretization of the coupled thermo-fluid dynamic problem. 
The Virtual Element method (VEM) is  a generalization of the Finite Element Method that takes inspiration from modern Mimetic Finite Difference schemes \cite{volley,autostoppisti,projectors}. It belongs to the family of polytopal element methods, i.e. finite element methods that can support  polygonal/polyhedral meshes.  Polytopal methods such as polygonal FEM, polygonal and Hybridizable Discontinuous Galerkin, Hybrid High Order methods, see e.g.\cite{N:2020, BdPD:2018, dPK:2018, AdP:2018, CdP:2020, LVY:2014, CFQ:2017,Antonietti-FSI-1,Antonietti-FSI-2}
have received substantial attention in the last years for fluid mechanic problems, thanks to their
flexibly in handling complicated geometries and  their ability preserving the underlying physical models' fundamental properties.
 The VEM has received substantial attention in the last ten years for a wide range problems stemming from of Engineering applications. We refer to the book \cite{Antonietti_BeiraodaVeiga_Manzini:2022} where the current state of the art on the Virtual Element Method is reported, collecting contributions from many of the most active researchers in this field. VEMs for Stokes and Navier-Stokes equations have been proposed and analyzed in \cite{ABMV:2014,BLV:2017,vacca:2018,BLV:2018,BDV:2020}; we refer, e.g., to \cite{berrone-supg,BdV-supg,BdV-oseen},  for the development of VEM for convection dominated problems.

%
%
%
%
%
%
%
%
%
%

The remaining part of the manuscript is organized as follows. Section \ref{sec:model_problem} describes the coupled thermo-fluid dynamic problem together with the theoretical results. Section \ref{sub:VEM} is devoted to the Virtual Elements discretization. The existence of discrete solutions is established in \ref{sec:exi} whereas the \textit{a priori} convergence analysis is presented in Section \ref{sec:conv}.  In Section \ref{sec:num} we present some numerical experiments to test the practical performance of the proposed VEM scheme. Finally, in Section\ref{sec:conclusions} we draw some conclusions.

We close this introduction with some standard notation. 
Let $\Omega \subset \R^2$ be the computational domain, 
we denote with $\xx = (x_1, \, x_2)$ the independent variable. 
With a usual notation   
the symbols $\gr$ and $\Delta$   
denote the gradient and the Laplacian for scalar functions, while 
$\dl $, $\Gr$, $\epseps$ and  $\dd$
denote the vector Laplacian,  the gradient, the symmetric gradient operator and the divergence operator,
whereas $\boldsymbol{\dd}$ denotes the vector valued divergence operator for tensor fields. 
%
Throughout the paper, we will follow the usual notation for Sobolev spaces
and norms \cite{Adams:1975}.
Hence, for an open bounded domain $\omega$,
the norms in the spaces $W^s_p(\omega)$ and $L^p(\omega)$ are denoted by
$\|{\cdot}\|_{W^s_p(\omega)}$ and $\|{\cdot}\|_{L^p(\omega)}$ respectively.
Norm and seminorm in $H^{s}(\omega)$ are denoted respectively by
$\|{\cdot}\|_{s,\omega}$ and $|{\cdot}|_{s,\omega}$,
while $(\cdot,\cdot)_{\omega}$ and $\|\cdot\|_{\omega}$ denote the $L^2$-inner product and the $L^2$-norm (the subscript $\omega$ may be omitted when $\omega$ is the whole computational
domain $\Omega$).

\section{Continuous problem}\label{sec:model_problem}

Let $\Omega \subseteq \R^2$ be a polygonal bounded domain with a Lipschitz-continuous boundary $\partial \Omega$.
We consider the following coupled thermo-fluid dynamic problem

\begin{equation}
\label{eq:pb primale}
\left\{
\begin{aligned}
&
\begin{aligned}
-   \boldsymbol{\dd}  (\nu(\theta) \,\epseps (\uu))  + (\Gr \uu ) \,\uu -  \nabla p &= \ff\qquad  & &\text{in $\Omega$,} \\
\dd \, \uu &= 0 \qquad & &\text{in $\Omega$,} \\
-  \dd(\kappa \, \nabla \theta) + \uu \cdot \nabla \theta &=  g \qquad  & &\text{in $\Omega$,} \\
\uu &= 0  \qquad & &\text{on $\partial \Omega$,}\\
\theta &= \td  \qquad & &\text{on $\partial \Omega$,}
\end{aligned}
\end{aligned}
\right.
\end{equation}
where the unknowns $\uu$, $p$ and $\theta$ represent the velocity, the pressure and the temperature of the fluid respectively; the function
$\nu$ is the temperature depending viscosity of the fluid, 
$\kappa$ is the thermal conductivity,
$\ff$ and $g$ stand for the the external volume source terms 
and $\td$ is the Dirichlet datum for the temperature.

Equation  \eqref{eq:pb primale} models the stationary flow of a viscous incompressible fluid (governed by a Navier-Stokes equation),  where the viscosity of the fluid depends on the temperature (governed by an elliptic equation).

For sake of simplicity we here  consider Dirichlet homogeneous boundary conditions for the velocity field, different boundary conditions can be treated as well. Moreover, the analysis of the three dimensional case could be developed with similar arguments. 
From now on, we assume that the model data satisfy the following assumptions.

\smallskip\noindent
\textbf{(A0) Data assumptions:} 
\begin{itemize}
\item  the conductivity $\kappa$ is positive and bounded, i.e. there exist $\kappa_*$, $\kappa^* > 0$ s.t.
\[
0 < \kappa_* \leq \kappa(\xx) \leq \kappa^* \quad \text{for all $\xx \in \Omega$;} 
\]
\item  the viscosity $\nu$ is positive, bounded and Lipschitz continuous, i.e. there exist $\nu_*$, $\nu^*$, $\clip > 0$ s.t.
\[
\begin{gathered}
0 < \nu_* \leq \nu(\theta) \leq \nu^* \qquad \text{for all $\theta \in \mathbb{R}$,}
\\
|\nu(\theta_1) - \nu(\theta_2)| \leq \clip |\theta_1 - \theta_2| \qquad \text{for all $\theta_1, \theta_2 \in \mathbb{R}$;}
\end{gathered}
\]
\item  the external loads satisfy $\ff \in [L^2(\Omega)]^2$, $g \in L^2(\Omega)$;

\item  the Dirichlet datum satisfy $\td \in C^0(\partial \Omega)$.

\end{itemize}

Notice that the regularity assumption on the boundary datum is required in order to define its  nodal interpolant in the discrete VEM setting (see \eqref{eq:thd}). 
Let us define the continuous spaces
\[
\begin{gathered}
\VV := \left[ H_0^1(\Omega) \right]^2, \quad 
Q := L^2_0(\Omega) = 
\left\{ q \in L^2(\Omega) \,\,\, \text{s.t.}  \,\,\, \int_{\Omega} q \,{\rm d}\Omega = 0 \right\}, 
\\
\W := H^1(\Omega), \quad 
\Wz := H^1_0(\Omega), \quad 
\Wd :=\{\sigma \in H^1(\Omega) \,\,\, \text{s.t.}  \,\,\, \sigma_{|\partial \Omega} = \td \},
\end{gathered}
\]
endowed with natural norms, and the forms
\begin{align}
\label{eq:forma av}
\av(\cdot; \, \cdot, \cdot) &\colon \W \times \VV \times \VV \to \R,    
&\quad
\av(\theta; \uu,  \vb) &:= 
\int_{\Omega} \nu(\theta)   \epseps(\uu) : \epseps (\vb) \, {\rm d}\Omega \,,
\\
\label{eq:forma as}
\as(\cdot, \cdot) &\colon \W \times \W \to \R,    
&\quad
\as(\theta,  \sigma) &:= 
\int_{\Omega} \kappa \, \nabla \sigma \cdot \nabla \theta \, {\rm d}\Omega \,,
\\
\label{eq:forma cv}
\cv(\cdot; \, \cdot, \cdot) &\colon \VV \times \VV \times \VV \to \R, 
&\quad
\cv(\ww; \, \uu, \vb) &:=  \int_{\Omega} ( \Gr \uu ) \, \ww \cdot \vb  \,{\rm d}\Omega \,,
\\
\label{eq:forma cs}
\cs(\cdot; \, \cdot, \cdot) &\colon \VV \times \W \times \W \to \R, 
&\quad
\cs(\uu; \, \theta, \sigma) &:=  \int_{\Omega}   \uu \cdot \nabla \theta \, \sigma  \,{\rm d}\Omega \,,
\\
\label{eq:forma b}
b(\cdot, \cdot) &\colon \VV \times Q \to \R, 
&\quad
b(\vb, q) &:=  \int_{\Omega}q\, \dd \vb \,{\rm d}\Omega \,.
\end{align}
Notice that under the assumptions  \textbf{(A0)} the forms are well defined, furthermore the following hold

\smallskip\noindent
\textbf{(P0) Stability properties of the continuous forms:} 
\\
$\bullet$ \,\, for any $\theta \in \W$ the bilinear form $\av(\theta; \, \cdot, \cdot)$ is coercive and continuous, i.e.
\[
\av(\theta; \, \vb, \vb) \geq \nu_* \vert \vb \vert_{1, \Omega}^2 \,,
\quad
\av(\theta; \, \uu, \vb) \leq \nu^* \vert \uu \vert_{1, \Omega} 
\vert \vb \vert_{1, \Omega}\,,
\qquad \text{for all $\uu$, $\vb \in \VV$;}
\]
$\bullet$ \,\, the bilinear form $\as(\cdot, \cdot)$ is coercive and continuous, i.e.
\[
\as(\sigma, \sigma) \geq \kappa_* \vert \sigma \vert_{1, \Omega}^2 \,,
\quad
\as(\theta, \sigma) \leq \kappa^* \vert \theta \vert_{1, \Omega} 
\vert \sigma \vert_{1, \Omega}\,,
\qquad \text{for all $\theta$, $\sigma \in \W$;}
\]
$\bullet$ \,\, the convective trilinear forms $\cv(\cdot; \, \cdot, \cdot)$ and $\cs(\cdot; \, \cdot, \cdot)$ are continuous with continuity constant $\ccon$, i.e.
\[
\begin{aligned}
\cv(\ww; \, \uu, \vb) &\leq  \ccon 
\vert \ww \vert_{1, \Omega} \vert \uu \vert_{1, \Omega} \vert \vb \vert_{1, \Omega} \,,
\qquad
& \text{for all $\ww$, $\uu$, $\vb \in \VV$,}
\\
\cs(\uu; \, \theta, \sigma) &\leq  \ccon
\vert \uu \vert_{1, \Omega} \Vert \theta \Vert_{1, \Omega} \Vert \sigma \Vert_{1, \Omega} \,,
\qquad
& \text{for all $\uu \in \VV$, $\theta$, $\sigma \in \W$;}
\end{aligned}
\]
\\
$\bullet$ \,\, the bilinear form $b(\cdot, \cdot)$ is continuous and realizes the inf-sup condition with inf–sup constant $\beta >0$, i.e. 
\[
\begin{aligned}
&b(\vb, q) \leq |\vb|_{1, \Omega} \Vert q \Vert_{0,\Omega} 
&\qquad &\text{for all $\vb \in \VV$ and $q \in Q$,}
\\
&\sup_{\vb \in \VV} \frac{b(\vb, q)}{|\vb|_{1, \Omega}} \geq \beta \Vert q \Vert_{0,\Omega} 
&\qquad &\text{for all $q \in Q$.}
\end{aligned}
\]

\smallskip\noindent
The variational formulation of Problem~\eqref{eq:pb primale} reads as follows:
\begin{equation}
\label{eq:pb variazionale}
\left\{
\begin{aligned}
& \text{find $(\uu, p, \theta) \in \VV \times Q \times \Wd$, such that} \\
&\begin{aligned}
\av(\theta, \, \uu, \vb) + \cv(\uu; \, \uu, \vb)  + b(\vb, p) &= (\ff, \vb) 
\qquad & &\text{for all $\vb \in \VV$,} \\
b(\uu, q) &= 0 \qquad & & \text{for all $q \in Q$,} \\
\as(\theta, \sigma) + \cs(\uu; \theta, \sigma) &= (g, \sigma) \qquad & &\text{for all $\sigma \in \Wz$.} 
\end{aligned}
\end{aligned}
\right.
\end{equation}
Let us introduce the kernel of the  bilinear form $b(\cdot,\cdot)$
that corresponds to the functions in $\VV$ with vanishing divergence, i.e.
\begin{equation*}
\label{eq:Z}
\ZZ := \{ \vb \in \VV \quad \text{s.t.} \quad \dd \, \vb = 0  \}\,.
\end{equation*}
Then, Problem~\eqref{eq:pb variazionale} can be formulated in the equivalent kernel form:
\begin{equation}
\label{eq:pb variazionale ker}
\left\{
\begin{aligned}
& \text{find $(\uu, \theta) \in \ZZ \times \Wd$, such that} \\
&\begin{aligned}
\av(\theta, \, \uu, \vb) + \cv(\uu; \, \uu, \vb)   &= (\ff, \vb) 
\qquad & &\text{for all $\vb \in \ZZ$,} \\
\as(\theta, \sigma) + \cs(\uu; \theta, \sigma) &= (g, \sigma) \qquad & &\text{for all $\sigma \in \Wz$.} 
\end{aligned}
\end{aligned}
\right.
\end{equation}
We can observe by a direct computation, that, for a fixed $\ww \in \ZZ$, the bilinear forms $\cv(\ww; \, \cdot , \cdot)$ and  $\cs(\ww; \, \cdot , \cdot)$ are skew symmetric, i.e. 
\begin{equation}
\label{eq:skew}
\begin{aligned}
&\cv(\ww; \uu, \vb) = -\cv(\ww; \vb, \uu) \qquad &\text{for all $\uu$, $\vb \in \VV$,}
\\
&\cs(\ww; \theta, \sigma) = -\cs(\ww; \sigma, \theta) \qquad &\text{for all $\theta$, $\sigma \in \W$.}
\end{aligned}
\end{equation}
Therefore, for $\ww \in \ZZ$, the forms $\cv(\ww; \, \cdot,  \cdot)$ and $\cs(\ww; \, \cdot,  \cdot)$  can be equivalently defined as follows
\begin{align}
\label{eq:forma cvskew}
&\cvskew(\ww; \, \uu,  \vb) := \frac{1}{2} \bigl(\cv(\ww; \, \uu, \vb) - \cv(\ww; \, \vb, \uu) \bigr)
\qquad &\text{for all $\uu$, $\vb \in \VV$,}
\\
\label{eq:forma csskew}
&\csskew(\ww; \, \theta,  \sigma) := \frac{1}{2} \bigl (\cs(\ww; \, \theta, \sigma) - \cs(\ww; \,  \sigma, \theta) \bigr)
\qquad &\text{for all $\theta$, $\sigma \in \W$.}
\end{align}
However, at the discrete level $\cv(\ww; \, \cdot ,  \cdot)$ and $\cvskew(\ww; \, \cdot , \cdot)$ (resp. $\cs(\ww; \, \cdot ,  \cdot)$ and $\csskew(\ww; \, \cdot , \cdot)$) will lead to different bilinear forms, in general.

In the context of the analysis of incompressible flows, it is useful to introduce the concept of Helmholtz--Hodge projector (see for instance \cite[Lemma~2.6]{john-linke-merdon-neilan-rebholz:2017} and \cite[Theorem~3.3]{GLS:2019}).
For every $\ww \in [L^2(\Omega)]^2$ there exist $\ww_0 \in H({\rm div}; \, \Omega)$
and $\zeta \in H^1(\Omega) / \R$ such that
\begin{equation}
\label{eq:helmholtz-hodge}
\ww = \ww_0 + \nabla \, \zeta,
\end{equation}
where $\ww_0$ is $L^2$-orthogonal to the gradients, that is $(\ww_0 , \, \nabla \phi) = 0$
for all $\phi \in H^1(\Omega)$ (which implies, in particular, that $\ww_0$ is solenoidal, i.e. ${\rm div} \, \ww_0 = 0$).
The orthogonal decomposition \eqref{eq:helmholtz-hodge} is unique and is called Helmholtz--Hodge decomposition, and $\mathcal{P}(\ww) :=\ww_0$ is the Helmholtz--Hodge projector of $\ww$.

Combining the argument in \cite[Theorem~2.2]{Bernardi:2015} and the definition of  Helmholtz--Hodge projector, the following existence result holds.

\begin{theorem}
\label{thm:exi}
Under the data assumptions \cfan{(A0)}, Problem \eqref{eq:pb variazionale} admits at least a solution $(\uu, p, \theta) \in \VV \times Q \times \Wd$. Moreover the solution satisfies the bound
\begin{equation}
\label{eq:bound cont}
\vert \uu \vert_{1, \Omega}^2 + \Vert \theta \Vert_{1, \Omega}^2 \leq \cest^2
\left( \vert \mathcal{P}(\ff) \vert_{-1, \Omega}^2  + 
\vert g \vert_{-1, \Omega}^2 +
\Vert \td \Vert_{1/2, \partial \Omega}^2 \right)
\end{equation}
where the constant $\cest$ depends on the domain $\Omega$ and on the constants $\kappa_*$, $\kappa^*$ and $\nu_*$ in the data assumptions \cfan{(A0)}.
\end{theorem}

%
%
%
%
%
%
%
\noindent
Assuming suitable bounds on the data of the problem and on the velocity solution it is possible to establish the following uniqueness result \cite[Proposition 2.3]{Bernardi:2015}

\begin{theorem}
\label{thm:uni}
Under the data assumptions \cfan{(A0)}, assume moreover that there exist two positive constants 
$C_{\rm data}$ and $C_{\rm sol}$ depending on the domain $\Omega$ and on the constants $\kappa_*$, $\kappa^*$ and $\nu_*$ in the data assumptions \cfan{(A0)} s.t.
\begin{itemize}
\item  [i)] the data of the problem satisfies
\begin{equation}
\label{eq:cdata}
C_{\rm data}^2 \left( \vert \mathcal{P}(\ff) \vert_{-1, \Omega}^2  + 
\vert g \vert_{-1, \Omega}^2 +
\Vert \td \Vert_{1/2, \partial \Omega}^2 \right) < 1 \,,
\end{equation}
\item [ii)] Problem \eqref{eq:pb variazionale} admits a solution $(\uu, p, \theta)$ with $\uu \in W^{1}_q(\Omega)$ where $q > 2$ and the following bound holds
\quad
\begin{equation}
\label{eq:csol}
C_{\rm sol} \, \clip \, \vert \uu \vert_{[W^{1}_q(\Omega)]^2} < 1 \,,
\end{equation}
\end{itemize}
then this solution is unique.
\end{theorem}

\section{Virtual Elements discretization}
\label{sub:VEM}
\subsection{Notation and preliminaries}
\label{sec:notations}

We now introduce some basic tools and notations useful in the construction and the theoretical analysis of Virtual Element Methods.

Let $\{\Omega_h\}_h$ be a sequence of decompositions of the domain $\Omega \subset \R^2$ into general polytopal elements $E$ 
where $h := \sup_{E \in \Omega_h} h_{E}$. 
%
%
We suppose that $\{\Omega_h\}_h$ fulfils the following assumption.

\noindent
\textbf{(A1) Mesh assumption:}
\\
there exists a positive constant $\rho$ such that for any $E \in \{\Omega_h\}_h$ 
\begin{itemize}
\item $E$ is star-shaped with respect to a ball $B_E$ of radius $ \geq\, \rho \, h_E$;
\item any edge $e$ of $E$ has length  $ \geq\, \rho \, h_E$.
\end{itemize}
We remark that the hypotheses above, though not too restrictive in many practical cases, could possibly be further relaxed, combining the present analysis with the studies in~\cite{BLR:2017,brenner-guan-sung:2017,brenner-sung:2018}.

Using standard VEM notations, for $n \in \N$ and  $s \in \R^+$ and $p=1, \dots, +\infty$
let us introduce the spaces:
\begin{itemize}
\item $\Pk_n(\omega)$: the set of polynomials on $\omega \subset \Omega$ of degree $\leq n$  (with $\Pk_{-1}(\omega)=\{ 0 \}$),
\item $\Pk_n(\Omega_h) := \{q \in L^2(\Omega) \quad \text{s.t} \quad q|_{E} \in  \Pk_n(E) \quad \text{for all $E \in \Omega_h$}\}$,
\item $W^s_p(\Omega_h) := \{v \in L^2(\Omega) \quad \text{s.t} \quad v|_{E} \in  W^s_p(E) \quad \text{for all $E \in \Omega_h$}\}$ 
\end{itemize}
equipped with the broken norm and seminorm
\[
\begin{aligned}
\|v\|^p_{W^s_p(\Omega_h)} &:= \sum_{E \in \Omega_h} \|v\|^p_{W^s_p(E)}\,,
&\quad
|v|^p_{W^s_p(\Omega_h)} &:= \sum_{E \in \Omega_h} |v|^p_{W^s_p(E)}\,, 
&\qquad \text{if $1 \leq p < \infty$,}
\\
\|v\|^p_{W^s_p(\Omega_h)} &:= \max_{E \in \Omega_h} \|v\|^p_{W^p(E)}\,,
&\quad
|v|^p_{W^s_p(\Omega_h)} &:= \max_{E \in \Omega_h} |v|^p_{W^s_p(E)}\,, 
&\qquad \text{if $p= \infty$.}
\end{aligned}
\]
Let $E \in \Omega_h$, we denote with $h_{E}$ the diameter, with $|E|$ the area, with $\xx_{E} = (x_{E, 1}, x_{E, 2})$ the centroid. 
A natural basis associated with the space $\Pk_n(E)$ is the set of normalized
monomials
\begin{equation*}
\label{eq:M_n}
\M_n(E) := \left\{ 
m_{\boldsymbol{\alpha}},
\,\,\,  \text{with} \,\,\,
|\boldsymbol{\alpha}| \leq n
\right\}
\end{equation*}
where, for any multi-index $\boldsymbol{\alpha} = (\alpha_1, \alpha_2) \in \N^2$
\[
m_{\boldsymbol{\alpha}} :=
\prod_{i=1}^2  
\left(\frac{x_i - x_{E, i}}{h_{E}} \right)^{\alpha_i}
\qquad \text{and} \qquad
|\boldsymbol{\alpha}|:= \sum_{i=1}^2 \alpha_i \,.
\]
Moreover for any $m \leq n$ we denote with
\[
\Pkw_{n \setminus m}(E) = {\rm span}  
\left\{ 
m_{\boldsymbol{\alpha}},
\,\,\,  \text{with} \,\,\,
m +1 \leq |\boldsymbol{\alpha}| \leq n
\right\} \,.
\]
Furthermore, we introduce the following notations: 
let $\{\mathcal{X}^E\}_{E \in \Omega_h}$ be a family 
of forms 
$\mathcal{X}^E \colon \prod_{j=1}^{\ell} W_{p_j}^{s_j}(E) \to \R$,
then we define 
\begin{equation}
\label{eq:XG}
\mathcal{X}\colon \prod_{j=1}^{\ell} W_{p_j}^{s_j}(\Omega_h)\to \R \,,
\quad 
\mathcal{X}(u_1, \dots, u_{\ell}) := \sum_{E \in \Omega_h}\mathcal{X}^E(u_1, \dots, u_{\ell}) \,,
\end{equation}
for any $u_j \in W^{s_j}_{p_j}(\Omega_h)$ and $j=1, \dots, \ell$.

For any $E$,  
let us define the following polynomial projections:
\begin{itemize}
\item the $\boldsymbol{L^2}$\textbf{-projection} $\Pi_n^{0, E} \colon L^2(E) \to \Pk_n(E)$, given by
\begin{equation}
\label{eq:P0_k^E}
\int_{E} q_n (v - \, {\Pi}_{n}^{0, E}  v) \, {\rm d} E = 0 \qquad  \text{for all $v \in L^2(E)$  and $q_n \in \Pk_n(E)$,} 
\end{equation} 
with obvious extension for vector functions $\Pi^{0, E}_{n} \colon [L^2(E)]^2 \to [\Pk_n(E)]^2$ and 
tensor functions $\boldsymbol{\Pi}^{0, E}_{n} \colon [L^2(E)]^{2 \times 2} \to [\Pk_n(E)]^{2\times 2}$;

\item the $\boldsymbol{H^1}$\textbf{-seminorm projection} ${\Pi}_{n}^{\nabla,E} \colon H^1(E) \to \Pk_n(E)$, defined by 
\begin{equation*}
\label{eq:Pn_k^E}
\left\{
\begin{aligned}
& \int_{E} \gr  \,q_n \cdot \gr ( v - \, {\Pi}_{n}^{\nabla,E}   v)\, {\rm d} E = 0 \quad  \text{for all $v \in H^1(E)$ and  $q_n \in \Pk_n(E)$,} \\
& \int_{\partial E}(v - \,  {\Pi}_{n}^{\nabla, E}  v) \, {\rm d}s= 0 \, ,
\end{aligned}
\right.
\end{equation*}
with extension for vector fields $\Pi^{\nabla, E}_{n} \colon [H^1(E)]^2 \to [\Pk_n(E)]^2$.
\end{itemize}

In the following the symbol $\lesssim$ will denote a bound up to a generic positive constant,
independent of the mesh size $h$, but which may depend on 
$\Omega$, on the ``polynomial'' order of the
method $k$ and on the regularity constant appearing in the mesh assumption \textbf{(A1)}.

We finally recall the following well know useful results:

\begin{itemize}

\item Poincar\'e inequality \cite[Theorem 1.3.3]{quarteroni-valli:book} 
\begin{equation}
\label{eq:poincare}
\Vert \phi \Vert_{1, \Omega} \lesssim \vert \phi \vert_{1, \Omega}
\qquad \text{for any $\phi \in H^1_0(\Omega)$;}
\end{equation}

\item Sobolev embedding $H^1(\Omega) \subset L^p(\Omega)$ \cite[Theorem 1.3.4]{quarteroni-valli:book}: let $2 \leq p <\infty$, then
\begin{equation}
\label{eq:sobemb}
\Vert \phi \Vert_{L^p(\Omega)} \lesssim \Vert \phi \Vert_{1, \Omega}
\qquad \text{for any $\phi \in H^1(\Omega)$;}
\end{equation}

\item Polynomial inverse estimate \cite[Theorem 4.5.11]{brenner-scott:book}: 
let $1 \leq q,p \leq \infty$, then for any $E \in \Omega_h$
\begin{equation}
\label{eq:inverse}
\Vert  p_n \Vert_{L^q(E)} \lesssim h_E^{2/q - 2/p} \Vert  p_n \Vert_{L^p(E)}
\quad \text{for any $p_n \in \Pk_n(E)$;}
\end{equation}

\item Bramble-Hilbert Lemma \cite[Lemma 4.3.8]{brenner-scott:book}: let $0\leq t \leq s \leq n+1$, and $1 \leq q,p \leq \infty$ such that $s - 2/p > t - 2/q$, then for any $E \in \Omega_h$
\begin{equation}
\label{eq:bramble}
\vert  \phi - \Pi^{0,E}_n \phi \vert_{W^t_q(E)} \lesssim 
h_E^{s-t + 2/q - 2/p} \vert  \phi \vert_{W^s_p(E)}
\quad \text{for any $\phi \in W^s_p(E)$.}
\end{equation}
\end{itemize}



The present section is devoted to the construction of the proposed virtual elements scheme.
In Subsections \ref{sub:v-p spaces} and \ref{sub:t-space} we present the inf-sup stable divergence--free velocities-pressures pair of spaces and the temperatures spaces, respectively.
In Subsection \ref{sub:forms} we define the discrete computable forms. Finally in Subsection \ref{sub:vem problem} we show the virtual elements discretization of Problems \eqref{eq:pb variazionale} and \eqref{eq:pb variazionale ker}.

Let $k \geq 2$ be the  ``polynomial'' order of the method.
The lowest order case $k=1$ can be treated as well using a slightly different approach \cite{ABMV:2014}.
We recall that, in standard finite element fashion, the VEM spaces are first defined elementwise and
then assembled globally.
In the following we will denote by $E$ a general polygon having $\ell_e$ edges $e$, while 
$\nn_E$ will denote the unit vector  that is normal to $\partial E$ and outward with respect to $E$.

\subsection{Virtual Elements velocities space and pressure space}
\label{sub:v-p spaces}

In the present section we outline an overview of the divergence-free Virtual
Elements spaces for the Navier-Stokes equation \cite{BLV:2017,vacca:2018,BLV:2018}.

We consider on each polygonal element $E \in \Omega_h$ the ``enhanced'' virtual space 
\begin{equation}
\label{eq:v-loc}
\begin{aligned}
\VDL := \biggl\{  
\vb_h \in [C^0(\overline{E})]^2 \,\,\, \text{s.t.} \,\,\,
(i)& \,\,  
  \boldsymbol{\Delta}    \vb_h  +  \nabla s \in \xxp \Pk_{k-1}(E), 
\,\,\text{ for some $s \in L_0^2(E)$,} 
\\
(ii)&  
\,\,   \dd \, \vb_h \in \Pk_{k-1}(E) \,, 
\\
(iii) &
\,\,  {\vb_h}_{|e} \in [\Pk_k(e)]^2 \,\,\, \forall e \in \partial E, 
\\
(iv) &
\,\,   (\vb_h - \Pi^{\nabla,E}_k \vb_h, \, \xxp \, \widehat{p}_{k-1} )_E = 0
\,\,\, \text{$\forall \widehat{p}_{k-1} \in \widehat{\Pk}_{k-1 \setminus k-3}(E)$}
\,\,\biggr\} \,,
\end{aligned}
\end{equation}
where $\xxp = (x_2, -x_1)$.
We here summarize the main properties of the space $\VDL$
(we refer to \cite{BLV:2018} for a deeper analysis).

\begin{itemize}
\item [\textbf{(P1)}] \textbf{Polynomial inclusion:} $[\Pk_k(E)]^2 \subseteq \VDL$;

\item [\textbf{(P2)}] \textbf{Degrees of freedom:}
the following linear operators $\mathbf{D_{\boldsymbol{V}}}$ constitute a set of DoFs for $\VDL$:
\begin{itemize}
\item[$\mathbf{D_{\boldsymbol{V}}1}$] the values of $\vb_h$ at the vertexes of the polygon $E$,
\item[$\mathbf{D_{\boldsymbol{V}}2}$] the values of $\vb_h$ at $k-1$ distinct points of every edge $e \in \partial E$,
\item[$\mathbf{D_{\boldsymbol{V}}3}$] the moments of $\vb_h$ 
$$
\frac{1}{|E|}\int_E  \vb_h \cdot \boldsymbol{m}^{\perp} m_{\boldsymbol{\alpha}} \, {\rm d}E  
\qquad 
\text{for any $m_{\boldsymbol{\alpha}} \in \M_{k-3}(E)$,}
$$
where $\boldsymbol{m}^{\perp} := \frac{1}{h_E}(x_2 - x_{2,E}, -x_1 + x_{1,E})$,
\item[$\mathbf{D_{\boldsymbol{V}}4}$] the moments of $\dd \vb_h$ 
$$
\frac{h_E}{|E|}\int_E \dd \vb_h \, m_{\boldsymbol{\alpha}} \, {\rm d}E  
\qquad \text{for any $m_{\boldsymbol{\alpha}} \in \M_{k-1}(E)$ with $|\boldsymbol{\alpha}| > 0$;}
$$
\end{itemize}
\item [\textbf{(P3)}] \textbf{Polynomial projections:}
the DoFs $\mathbf{D_{\boldsymbol V}}$ allow us to compute the following linear operators:
\[
\P0 \colon \VDL \to [\Pk_{k}(E)]^2, \qquad
\PP0 \colon \Gr \VDL \to [\Pk_{k-1}(E)]^{2 \times 2} \,.
\]
\end{itemize} 
The global velocity space $\VDG$ is defined by gluing the local spaces with the obvious
associated sets of global DoFs:
\begin{equation}
\label{eq:v-glo}
\VDG := \{\vb_h \in \VV \quad \text{s.t.} \quad {\vb_h}_{|E} \in \VDL \quad \text{for all $E \in \Omega_h$} \}\,.
\end{equation}
The discrete pressure space $\QDG$ is given by the piecewise polynomial functions of degree $k-1$, i.e.
\begin{equation}
\label{eq:q-glo}
\QDG := \{q_h \in \QCG \quad \text{s.t.} \quad {q_h}_{|E} \in \Pk_{k-1}(E) \quad \text{for all $E \in \Omega_h$} \}\,.
\end{equation}
The couple of spaces $(\VDG, \, \QDG)$ is inf-sup stable \cite{BLV:2017}, and we denote with $\isd >0$ the inf-sup stability constant, i.e. $\isd$ is such that
\begin{equation}
\label{eq:infsup}
\sup_{\vb_h \in \VDG} \frac{b(\vb_h, q_h)}{\vert \vb_h\vert_{1, \Omega}} \geq \isd \|q_h\|_{0,E}
\qquad \text{for any $q_h \in \QDG$.}
\end{equation}
Let us introduce the discrete kernel
\begin{equation}
\label{eq:z-glo}
\ZDG := \{ \vb_h \in \VDG \quad \text{s.t.} \quad  b(\vb_h, q_h) = 0 \quad \text{for all $q_h \in \QDG$}\}
\end{equation}
then recalling $(ii)$ in \eqref{eq:v-loc} and \eqref{eq:q-glo}, the following kernel inclusion holds
\begin{equation}
\label{eq:kernel}
\ZDG \subseteq \ZCG \,,
\end{equation}
i.e. the functions in the discrete kernel are exactly divergence-free.

\subsection{Virtual Elements space for the temperatures}
\label{sub:t-space}
In this section, we briefly introduce the $H^1$-conforming virtual space for the temperatures, that consists on the so-called nodal ``enhanced'' virtual space \cite{projectors}.
We thus consider on each element $E \in \Omega_h$  the space
\begin{equation}
\label{eq:w-loc}
\begin{aligned}
\WDL := \biggl\{  
\sigma_h \in C^0(\overline{E}) \,\,\, \text{s.t.} \,\,\,
(i)& \,  
\,  \Delta \sigma_h \in \Pk_{k}(E), 
\\
(ii) &
\, \, {\sigma_h}_{|e} \in \Pk_k(e) \,\,\, \forall e \in \partial E, 
\\
(iii) &
\, \, (\sigma_h - \Pi^{\nabla,E}_k \sigma_h, \, \widehat{p}_{k} )_E = 0
\,\,\, \text{$\forall \widehat{p}_{k} \in \widehat{\Pk}_{k \setminus k-2}(E)$}
\,\,\biggr\} \,.
\end{aligned}
\end{equation}
We here summarize the main properties of the space $\WDL$
(see \cite{projectors}).

\begin{itemize}
\item [\textbf{(P4)}] \textbf{Polynomial inclusion:} $\Pk_k(E) \subseteq \WDL$;

\item [\textbf{(P5)}] \textbf{Degrees of freedom:}
the following linear operators $\mathbf{D_{\boldsymbol \Sigma}}$ constitute a set of DoFs for $\WDL$:
\begin{itemize}
\item[$\mathbf{D_{\boldsymbol \Sigma}1}$] the values of $\sigma_h$ at the vertexes of the polygon $E$,
\item[$\mathbf{D_{\boldsymbol \Sigma}2}$] the values of $\sigma_h$ at $k-1$ distinct points of every edge $e \in \partial E$,
\item[$\mathbf{D_{\boldsymbol \Sigma}3}$] the moments of $\sigma_h$ 
$$
\frac{1}{|E|}\int_E  \sigma_h  \,  m_{\boldsymbol{\alpha}} \, {\rm d}E  
\qquad \text{for any $m_{\boldsymbol{\alpha}} \in \M_{k-2}(E)$;}
$$
\end{itemize}
\item [\textbf{(P6)}] \textbf{Polynomial projections:}
the DoFs $\mathbf{D_{\boldsymbol \Sigma}}$ allow us to compute the following linear operators:
\[
\P0 \colon \WDL \to \Pk_{k}(E), \qquad
\Pp0 \colon \nabla \WDL \to [\Pk_{k-1}(E)]^2 \,.
\]
\end{itemize} 
The global velocity space $\WDG$ is defined by gluing the local spaces with the obvious
associated sets of global DoFs:
\begin{equation*}
\label{eq:w-glo-g}
\WDG := \{\sigma_h \in \W \,\,\, \text{s.t.} \,\,\, {\sigma_h}_{|E} \in \WDL \,\,\, \text{for all $E \in \Omega_h$} \}\,,
\end{equation*}
with its homogeneous boundary counterpart
\begin{equation}
\label{eq:w-glo-z}
\WDZG := \{\sigma_h \in \WDG \,\,\, \text{s.t.} \,\,\, {\sigma_h}_{|\partial \Omega} = 0 \}\,.
\end{equation}
In order to treat non-homogeneous Dirichlet data, 
accordingly to the data assumptions \textbf{(A0)},
let us consider the  ``nodal'' interpolant $\tdh$ of $\td$ defined by:
\begin{equation}
\label{eq:thd}
\begin{gathered}
\tdh \in C^0(\partial \Omega) 
\quad \text{s.t.} \quad
{\tdh}_{|e} \in \Pk_{k}(e) 
\quad \text{for any edge $e \subset \partial \Omega$,} 
\\
\tdh(\xx_i) = \td(\xx_i) 
\qquad 
\text{for any node $\xx_i \in e$, for any edge $e \subset \partial \Omega$.}
\end{gathered} 
\end{equation}
Then the discrete counterpart of the space $\Wd$ is defined as follows:
\begin{equation}
\label{eq:w-glo-d}
\WDDG := \{\sigma_h \in \WDG \,\,\, \text{s.t.} \,\,\, {\sigma_h}_{|\partial \Omega} = \tdh \}\,.
\end{equation}

\subsection{Virtual Element forms}
\label{sub:forms}

The next step in the construction of our method is the definition of discrete versions of the continuous second order forms in \eqref{eq:forma av} and \eqref{eq:forma as},
of the convective forms in \eqref{eq:forma cv} and \eqref{eq:forma cs} and the approximation of the right-hand side terms.
It is clear that for an arbitrary function in $\VDG$ or $\WDG$ the forms are not computable since the discrete functions are not known in closed form.
Therefore, following the usual procedure in the VEM setting, we need to construct discrete forms that are computable by {means of the only knowledge of} the DoFs.

In the light of properties \textbf{(P3)} and \textbf{(P6)} we define the computable discrete forms.

\noindent $\bullet$ \, Discrete second order forms (cf.  \eqref{eq:forma av} and \eqref{eq:forma as})
\begin{align}
\label{eq:forma avhl}
\avhl(\theta_h; \uu_h,  \vb_h) &:= 
\int_{E} \nu\bigl(\P0 \theta_h \bigr) \,  \PP0 \epseps(\uu_h) : \PP0\epseps (\vb_h) \, {\rm d}E
+\svl(\theta_h; \uu_h, \, \vb_h) ,
\\
\label{eq:forma ashl}
\ashl(\theta_h,  \sigma_h) &:= 
\int_{E} \kappa \, \Pp0\nabla \theta_h \cdot \Pp0 \nabla \sigma_h \, {\rm d}E 
+ \ssl(\theta_h, \, \sigma_h) ,
\end{align}
where the VEM stabilizing terms are given by
\begin{align}
\label{eq:svstab}
\svl(\theta_h; \uu_h, \, \vb_h) &:= \nu(\Pi^{0,E}_0 \theta_h) \,
S_V^E \bigl( (I - \P0) \uu_h, \, (I - \P0 ) \vb_h \bigr) \,,
\\
\label{eq:ssstab}
\ssl(\theta_h, \, \sigma_h) &:= \bigl(\Pi^{0,E}_0 \kappa \bigr) \,
S_T^E \bigl( (I - \P0) \theta_h, \, (I - \P0 ) \sigma_h \bigr) \,,
\end{align}
and where $S_V^E(\cdot, \cdot) \colon \VDL \times \VDL \to \R$ and $S_E^E(\cdot, \cdot) \colon \WDL \times \WDL \to \R$ are computable symmetric discrete forms satisfying 
\begin{equation}
\label{eq:stab}
\begin{gathered}
\vert \vb_h \vert^2_{1,E} \lesssim S_V^E(\vb_h, \vb_h) \lesssim \vert \vb_h \vert^2_{1,E} 
\quad \text{for all $\vb \in \VDL \cap \ker \P0$,}
\\
\vert  \sigma_h \vert^2_{1,E} \lesssim S_T^E(\sigma_h, \sigma_h) \lesssim \vert  \sigma_h \vert^2_{1,E} 
\quad \text{for all $\sigma_h \in \WDL \cap \ker \P0$.}
\end{gathered}
\end{equation}
Many examples can be found in the VEM literature 
\cite{volley,BDR:2017}.
In the present paper we consider the so-called \texttt{dofi-dofi} stabilization defined as follows: let $\vec{\uu}_h$, $\vec{\vb}_h$ and $\vec{\theta}_h$, $\vec{\sigma}_h$  denote the real valued vectors containing the values of the local degrees of freedom associated to $\uu_h$, $\vb_h$ in the space $\VDG$ and to $\theta_h$ and $\sigma_h$ in the space $\WDG$ respectively, then
\begin{equation*}
\label{eq:dofi}
S_V^E(\uu_h, \vb_h) :=\vec{\uu}_h \cdot \vec{\vb}_h \,,
\qquad
S_T^E(\theta_h, \sigma_h) :=\vec{\theta}_h \cdot \vec{\sigma}_h \,.
\end{equation*}

\noindent $\bullet$ \, Discrete convective forms (cf. \eqref{eq:forma cv} and \eqref{eq:forma cs})
\begin{align}
\label{eq:forma cvhl}
c_{V,h}^E(\ww_h; \, \uu_h, \vb_h) &:= 
\int_{E} \bigl[\PP0( \Gr \uu_h ) \, \P0 \ww_h \bigr] \cdot \P0 \vb_h  \,{\rm d}E,
\\
\label{eq:forma cshl}
c_{T,h}^E(\uu_h; \, \theta_h, \sigma_h) &:=  \int_{E}  \bigl(\P0 \uu_h \cdot \Pp0\nabla \theta_h \bigr) \, \P0 \sigma_h  \,{\rm d}E \,,
\end{align}
and their skew-symmetric formulation (cf. \eqref{eq:forma cvskew} and \eqref{eq:forma csskew})
\begin{align}
\label{eq:forma cvhlskew}
&\cvhl(\ww_h; \, \uu_h,  \vb_h) := \frac{1}{2} \bigl(c_{V,h}^E(\ww_h; \, \uu_h, \vb_h) - c_{V,h}^E(\ww_h; \, \vb_h, \uu_h) \bigr) \,,
\\
\label{eq:forma cshlskew}
&\cshl(\ww_h; \, \theta_h,  \sigma_h) := \frac{1}{2} \bigl (c_{T,h}^E(\ww_h; \, \theta_h, \sigma_h) - c_{T,h}^E(\ww_h; \, \sigma_h, \theta_h) \bigr) \,.
\end{align}

\noindent $\bullet$ \, \, Discrete external forces
\begin{equation}
\label{eq:forma fh gh}
{\ff_h}_{|E} := \P0 \ff \,,
\quad
{g_h}_{|E} := \P0 g \,,
\qquad \text{for all $E \in \Omega_h$.}
\end{equation}

\noindent
The global forms, defined accordingly to the notation \eqref{eq:XG}, satisfy the following.

\smallskip\noindent
\textbf{(P7) Stability properties of the discrete forms:} 
\\
$\bullet$ \,\, for any $\theta \in \W$ the bilinear form $\avh(\theta; \, \cdot, \cdot)$ is coercive and continuous, i.e.
\[
\avh(\theta; \, \vb_h, \vb_h) \geq \avs \nu_* \vert \vb_h \vert_{1, \Omega}^2 \,,
\quad
\avh(\theta; \, \uu_h, \vb_h) \leq \avS \nu^* \vert \uu_h \vert_{1, \Omega} 
\vert \vb_h \vert_{1, \Omega}\,,
\]
for all $\uu_h$, $\vb_h \in \VDG$;

\noindent
$\bullet$ \,\, the bilinear form $\ash(\cdot, \cdot)$ is coercive and continuous, i.e.
\[
\ash(\sigma_h, \sigma_h) \geq \ats \kappa_* \vert \sigma_h \vert_{1, \Omega}^2 \,,
\quad
\as(\theta_h, \sigma_h) \leq \atS \kappa^* \vert \theta_h \vert_{1, \Omega} 
\vert \sigma_h \vert_{1, \Omega}\,,
\]
for all $\theta_h$, $\sigma_h \in \WDG$;

\noindent
$\bullet$ \,\, the convective trilinear forms $\cvh(\cdot; \, \cdot, \cdot)$ and $\csh(\cdot; \, \cdot, \cdot)$ are continuous, i.e.
\[
\begin{aligned}
\cvh(\ww; \, \uu, \vb) &\leq  \cconh 
\vert \ww \vert_{1, \Omega} \vert \uu \vert_{1, \Omega} \vert \vb \vert_{1, \Omega} \,,
\qquad
& \text{for all $\ww$, $\uu$, $\vb \in \VV$,}
\\
\csh(\uu; \, \theta, \sigma) &\leq  \cconh  
\vert \uu \vert_{1, \Omega} \Vert \theta \Vert_{1, \Omega} \Vert \sigma \Vert_{1, \Omega} \,,
\qquad
& \text{for all $\uu \in \VV$, $\theta$, $\sigma \in \W$,}
\end{aligned}
\]
where $\avs$, $\avS$, $\ats$, $\atS$, $\cconh$ are positive constants depending on the domain $\Omega$, on the ``polynomial'' order of the
method $k$ and on the regularity constant appearing in the mesh assumption \textbf{(A1)}.

\subsection{Virtual Element problem}
\label{sub:vem problem}

{ Having in mind} the spaces \eqref{eq:v-glo}, \eqref{eq:q-glo}, \eqref{eq:w-glo-d} and \eqref{eq:w-glo-z},  the discrete forms  \eqref{eq:forma avhl}, \eqref{eq:forma ashl}, \eqref{eq:forma cvhlskew}, \eqref{eq:forma cshlskew}, the form \eqref{eq:forma b}, the discrete loading terms \eqref{eq:forma fh gh}, the virtual element discretization of Problem \eqref{eq:pb variazionale} is given by

\begin{equation}
\label{eq:pb vem}
\left\{
\begin{aligned}
& \text{find $(\uu_h, p_h, \theta_h) \in \VDG \times \QDG \times \WDDG$, such that} \\
&\begin{aligned}
\avh(\theta_h, \, \uu_h, \vb_h) + \cvh(\uu_h; \, \uu_h, \vb_h)  + b(\vb_h, p_h) &= (\ff_h, \vb_h) 
\quad & &\text{$\forall \vb_h \in \VDG$,} \\
b(\uu_h, q_h) &= 0 \quad & & \text{$\forall  q_h \in \QDG$,} \\
\ash(\theta_h, \sigma_h) + \csh(\uu_h; \theta_h, \sigma_h) &= (g_h, \sigma_h) \qquad & &\text{$\forall \sigma_h \in \WDZG$.} 
\end{aligned}
\end{aligned}
\right.
\end{equation}
Recalling the definition of { the}discrete kernel $\ZDG$ in \eqref{eq:z-glo}, the previous problem can be also written in the kernel formulation, i.e.
\begin{equation}
\label{eq:pb vem ker}
\left\{
\begin{aligned}
& \text{find $(\uu_h, \theta_h) \in \ZZ_h \times \WDDG$, such that} \\
&\begin{aligned}
\avh(\theta_h, \, \uu_h, \vb_h) + \cvh(\uu_h; \, \uu_h, \vb_h)   &= (\ff_h, \vb_h) 
\quad & &\text{$\forall \vb_h \in \ZZ_h$,} \\
\ash(\theta_h, \sigma_h) + \csh(\uu_h; \theta_h, \sigma_h) &= (g_h, \sigma_h) \quad & &\text{$\forall \sigma_h \in \WDZG$.} 
\end{aligned}
\end{aligned}
\right.
\end{equation}

\section{Existence of discrete solutions}
\label{sec:exi}

In the present section we establish the existence of discrete solutions.

\begin{proposition}
\label{prp:existence}
Under the data assumptions \cfan{(A0)}, assume moreover that there exists $\cexi$ depending on the domain $\Omega$ and the shape regularity constant in assumption \cfan{(A1)} s.t.
\begin{equation}\label{aux:Prop:5:1}
\Vert \tdh\Vert_{1/2, \partial \Omega} 
\leq \cexi  \min\{\nu_*, \kappa_* \}
\end{equation}
then  Problem \eqref{eq:pb vem} admits at least a solution $(\uu_h, p_h, \theta_h) \in \ZDG \times \QDG \times \WDDG$.
Moreover, the following bound holds
\begin{equation}
\label{eq:bound disc}
\vert \uu_h \vert_{1, \Omega}^2 + \Vert \theta_h  \Vert_{1, \Omega}^2 \leq \cesth^2
\left( \Vert \ff_h \Vert_{\ZDG^*}^2  + 
\vert g \vert_{-1, \Omega}^2 +
\Vert \td \Vert_{1/2, \partial \Omega}^2 \right)
\end{equation}
where the constant $\cesth$ depends on the domain $\Omega$ and on the constants $\kappa_*$, $\kappa^*$ and $\nu_*$ in the data assumptions \cfan{(A0)} and the shape regularity constant in assumption \cfan{(A1)}.
\end{proposition}

\begin{proof}
Let $\tparh \in \WDDG$ be the solution of the following auxiliary problem
\begin{equation*}
\label{eq:tparh}
\left\{
\begin{aligned}
& \text{find $\tparh \in \WDDG$, such that} \\
&\begin{aligned}
\int_{\Omega} \nabla \tparh \cdot \nabla \sigma_h \, {\rm d}\Omega
   &= 0 
\quad & &\text{$\forall \sigma_h \in \WDZG$.}
\end{aligned}
\end{aligned}
\right.
\end{equation*}
Then, by Lax-Milgram Lemma,  $\tparh$ satisfies the bound
\begin{equation}
\label{eq:tpar-bouh}
\Vert \tparh\Vert_{1, \Omega} \leq C_\Omega \, \Vert\tdh\Vert_{1/2, \partial \Omega} \,.
\end{equation}
Notice that Problem \eqref{eq:pb vem ker} can be equivalently formulated as follows

\begin{equation}
\label{eq:pb vem kerZ}
\left\{
\begin{aligned}
& \text{find $(\uu_h, \tinth) \in \ZDG \times \WDZG$, such that} \\
&\begin{aligned}
\avh(\tinth + \tparh, \, \uu_h, \vb_h) + \cvh(\uu_h; \, \uu_h, \vb_h)  &= (\ff_h, \vb_h) 
 & &\text{$\forall \vb_h \in \ZDG$,} \\
\ash(\tinth + \tparh, \sigma_h) + \csh(\uu_h; \tinth + \tparh, \sigma_h) &= (g_h, \sigma_h)  & &\text{$\forall \sigma_h \in \WDZG$.} 
\end{aligned}
\end{aligned}
\right.
\end{equation}
We now prove that Problem \eqref{eq:pb vem kerZ} admits solutions.
Let $\beth := \ZDG \times \WDZG$ and, for any $(\uu_h, \tinth) \in \beth$ let $\mathcal{A}^{(\uu_h, \tinth)}\colon \beth \to \R$  given  by
\begin{equation}
\label{eq:prpex1}
\begin{aligned}
\mathcal{A}^{(\uu_h, \tinth)}(\vb_h, \sigma_h)
&:= 
\avh(\tinth + \tparh, \, \uu_h, \vb_h) + 
\cvh(\uu_h; \, \uu_h, \vb_h) - (\ff_h, \vb_h)  \\
& \quad
+ \ash(\tinth + \tparh, \sigma_h)  +
\csh(\uu_h; \tinth + \tparh, \sigma_h) - (g_h, \sigma_h) 
\end{aligned}
\end{equation}
for all $(\vb_h, \sigma_h) \in \beth$.
Then employing property \textbf{(P7)} and the skew-symmetry of the convective forms $\cvh$ and $\csh$ (cf. \eqref{eq:forma cvhlskew} and \eqref{eq:forma cshlskew}) we obtain 
\begin{equation}
\label{eq:prpex2}
\begin{aligned}
\mathcal{A}^{(\uu_h, \tinth)}(\uu_h, \tinth)
& \geq 
\alpha_* \nu_* \vert \uu_h\vert_{1, \Omega}^2 
+ \beta_* \kappa_* \vert \tinth\vert_{1, \Omega}^2 
- \Vert \ff_h \Vert_{\ZDG^*}  \vert \uu_h\vert_{1, \Omega}
\\
&\qquad
- \vert g_h\vert_{-1, \Omega} \vert \tinth \vert_{-1, \Omega} 
+ \ash(\tparh, \tinth)  +
\csh(\uu_h; \tparh, \tinth)   
 \,.
\end{aligned}
\end{equation}
From property \textbf{(P7)} and \eqref{eq:tpar-bouh} it holds that
\begin{equation}
\label{eq:prpex3}
\begin{aligned}
\ash(\tparh, \tinth) &\geq 
-\ash(\tparh, \tparh)^{1/2} \,
\ash(\tinth, \tinth)^{1/2}
\geq -\beta^* \,\kappa^* \, C_\Omega \,  \Vert \tdh \Vert_{1/2, \Omega}  \,
\vert \tinth\vert_{1, \Omega} \,,
\\
\csh(\uu_h; \tparh, \tinth)   &\geq - \cconh 
\Vert \tparh \Vert_{1, \Omega}  \,
\vert \uu_h\vert_{1, \Omega}
\vert \tinth\vert_{1, \Omega} 
\geq - \frac{\cconh}{2} C_\Omega \,
\Vert \tdh \Vert_{1/2, \Omega}  
(\vert \uu_h\vert_{1, \Omega}^2 + \vert \tinth\vert_{1, \Omega}^2) \,.
\end{aligned}
\end{equation}
Therefore inserting \eqref{eq:prpex3} in \eqref{eq:prpex2}, we infer
\begin{equation}
\label{eq:prpex4}
\begin{aligned}
\mathcal{A}^{(\uu_h, \tinth)}(\uu_h, \tinth)&
\geq 
\left( \min\{\alpha_*, \beta_* \}\min\{\nu_*, \kappa_* \}  - \frac{1}{2}\cconh C_\Omega \,\Vert \tdh \Vert_{1/2, \Omega}   \right) (\vert \uu_h\vert_{1, \Omega}^2 + \vert \tinth\vert_{1, \Omega}^2) +
\\
& \quad
- \left( (\Vert \ff_h \Vert_{\ZDG^*}^2 + \vert g_h\vert_{-1, \Omega}^2)^{1/2} + \beta^*\, \kappa^*C_\Omega \, \Vert \tdh \Vert_{1/2, \Omega} \right)
(\vert \uu_h\vert_{1, \Omega}^2 + \vert \tinth\vert_{1, \Omega}^2)^{1/2}
 \,.
\end{aligned}
\end{equation}
Let us set
\[
\cexi \leq  \frac{\min\{\alpha_*, \beta_* \}}{\cconh C_\Omega } \,,
\]
then from { \eqref{aux:Prop:5:1}} and  \eqref{eq:prpex4} we get
\begin{equation}
\label{eq:prpex5}
\begin{aligned}
\mathcal{A}^{(\uu_h, \tinth)}(\uu_h, \tinth)&
\geq 
\frac{1}{2} \min\{\alpha_*, \beta_* \}\min\{\nu_*, \kappa_* \}   (\vert \uu_h\vert_{1, \Omega}^2 + \vert \tinth\vert_{1, \Omega}^2) +
\\
& \quad 
- \left( (\Vert \ff_h \Vert_{\ZDG^*}^2 + \vert g_h\vert_{-1, \Omega}^2)^{1/2} + \beta^* \,\kappa^* C_\Omega \,\Vert \tdh \Vert_{1/2, \Omega} \right)
(\vert \uu_h\vert_{1, \Omega}^2 + \vert \tinth\vert_{1, \Omega}^2)^{1/2}
 \,.
\end{aligned}
\end{equation}
Let 
\begin{equation}
\label{eq:rho}
\rho := 2\frac{(\Vert \ff_h \Vert_{\ZDG^*}^2 + \vert g_h\vert_{-1, \Omega}^2)^{1/2} +  \beta^* \,\kappa^* C_\Omega \, \Vert \tdh \Vert_{1/2, \Omega}}
{\min\{\alpha_*, \beta_* \}\min\{\nu_*, \kappa_* \}}
\end{equation}
and
$
\mathcal{S} := 
\{ (\vb_h, \sigma_h) \in \beth \quad \text{s.t.} \quad \vert \uu_h\vert_{1, \Omega}^2 + \vert \tinth\vert_{1, \Omega}^2 \leq \rho^2\}
$
then \eqref{eq:prpex5} implies 
\[
\mathcal{A}^{(\uu_h, \tinth)}(\uu_h, \tinth) \geq 0 \qquad 
\text{for any $(\uu_h, \tinth) \in \partial \mathcal{S}$.}
\]
Employing the fixed-point Theorem \cite[Chap. IV, Corollary 1.1]{girault-raviart:book2} there exists $(\uu_h, \tinth) \in \mathcal{S}$ s.t. 
$\mathcal{A}^{(\uu_h, \tinth)}(\vb_h, \sigma_h) = 0$ for all $(\vb_h, \sigma_h) \in \beth$, i.e. Problem \eqref{eq:pb vem kerZ} admits solution.
Let $\theta_h = \tinth + \tparh$, then $(\uu_h, \theta_h) \in \ZDG \times \WDZG$ is a solution of Problem \eqref{eq:pb vem ker}. 
Bound \eqref{eq:bound disc} follows from \eqref{eq:rho} and \eqref{eq:tpar-bouh} and the definition of $ \mathcal{S}$.

Finally, the existence of solutions $(\uu_h, p_h, \tinth) \in \VDG \times \QDG \times \WDZG$ of the equivalent Problem \eqref{eq:pb vem} is a direct consequence of \eqref{eq:infsup} (see for instance  \cite{boffi-brezzi-fortin:book}).
\end{proof}

\begin{remark}
\label{rm:existence}
Notice that, differently from the continuous problem, in order to establish the existence result for the discrete problem \eqref{aux:Prop:5:1} we need that the conductivity $\kappa$ and the viscosity $\nu$ are sufficiently large with respect to the Dirichlet datum $\theta_D$ (cf. bound \eqref{eq:pb vem}). 
A careful inspection of the proof of Theorem \ref{thm:exi} (see \cite[Theorem~2.2]{Bernardi:2015}) 
shows that the gap between the continuous and the discrete case is the different construction of the (skew-symmetric) convective forms $c_T(\uu; \, \cdot, \cdot)$ and $c_T^{\rm skew}(\uu_h; \, \cdot, \cdot)$.
\end{remark}

We now investigate whether Problem \eqref{eq:pb vem ker} admits a unique solution.
We introduce the analysis with the following lemma.

\begin{lemma}
\label{lm:lipshitz}
Let $\sigma_{h,1}$, $\sigma_{h,2} \in \WDDG$, let $\vb_h \in \VDG$ such that $\vb_h \in [W_1^p(\Omega_h)]^2$ with $p>2$.
There exists a positive constant $\cclip$ only depending on $\Omega$, $k$, and the shape regularity constant in \cfan{(A1)} such that for any $\ww_h\in \VDG$
\begin{equation}
\label{eq:cclip}
\begin{aligned}
\vert \avh(\sigma_{h,1}; \vb_h, \ww_h) -  \avh(\sigma_{h,2}; \vb_h, \ww_h) \vert 
\leq 
\cclip \clip 
\vert \sigma_{h,1}-\sigma_{h,2}\vert_{1, \Omega} \, 
\vert \vb_h \vert_{[W^1_q(\Omega)]^2} \,
\vert \ww_h \vert_{1, \Omega}   \,.
\end{aligned}
\end{equation} 
\end{lemma}

\begin{proof}
From definition of $\avh(\cdot, \cdot)$ (cf. \eqref{eq:forma avhl}) we have
\begin{equation}
\label{eq:lmlip1}
\begin{aligned}
\vert  \avh(\sigma_{h,1}; \vb_h, \ww_h) -  \avh(\sigma_{h,2}; \vb_h, \ww_h) \vert  
\leq
\sum_{E \in \Omega_h} 
\left \vert \svl(\sigma_{h,1}; \vb_h, \ww_h) - \svl(\sigma_{h,2}; \vb_h, \ww_h) \right \vert  +
\\
+ \sum_{E \in \Omega_h}  \int_{E} \bigl \vert \nu\bigl(\P0 \sigma_{h,1} \bigr)- \nu\bigl(\P0 \sigma_{h,2}  \bigr) \bigr \vert
\,  \bigl\vert \PP0\epseps (\vb_h) \bigr \vert 
\bigl\vert \PP0\epseps (\ww_h) \bigr \vert \, {\rm d}E  
:= \eta_1 + \eta_2 \,.
\end{aligned}
\end{equation}
Concerning the term $\eta_1$, from \eqref{eq:svstab} and \eqref{eq:stab}, and recalling property \textbf{(P0)}, we infer
\[
\begin{aligned}
\eta_1 
&\leq 
\sum_{E \in \Omega_h} 
\bigl \vert \nu(\Pi^{0,E}_0 \sigma_{h,1}) - \nu(\Pi^{0,E}_0 \sigma_{h,2}) \bigr \vert \,
\bigl \vert S_V^E \bigl( (I - \P0) \vb_h, \, (I - \P0 ) \ww_h \bigr) \bigr \vert
\\
&\lesssim \clip
\sum_{E \in \Omega_h} 
 \vert \Pi^{0,E}_0 (\sigma_{h,1} -  \sigma_{h,2}) \vert \,
\vert  (I - \P0) \vb_h \vert_{1,E} \, \vert(I - \P0 ) \ww_h \vert_{1,E} \,.
\end{aligned}
\]
Then employing the inverse estimate \eqref{eq:inverse} and the Bramble-Hilbert Lemma \eqref{eq:bramble}
\[
\begin{aligned}
\eta_1 
&\lesssim
\clip
\sum_{E \in \Omega_h} 
h_{ E}^{-2/p} \Vert \Pi^{0,E}_0 (\sigma_{h,1} -  \sigma_{h,2}) \Vert_{L^p(E)} \, 
h_E^{1 -2/q}\vert  \vb_h \vert_{[W^{1}_q(\Omega_h)]^2} \, \vert \ww_h \vert_{1,E} 
\end{aligned}
\]
where  $p$ is such that $1/p + 1/q = 1/2$. Therefore, since $\Pi^{0,E}_0$ is continuous with respect to the $L^p$-norm, from the H\"older inequality for sequences we infer
\[
\begin{aligned}
\eta_1 
&\lesssim
\clip
\sum_{E \in \Omega_h} 
\Vert \sigma_{h,1} -  \sigma_{h,2} \Vert_{L^p(E)} \, 
\vert  \vb_h \vert_{[W^1_q(E)]^2} \, \vert \ww_h \vert_{1,E} 
\\
&\lesssim
\clip
\biggl( \sum_{E \in \Omega_h}  \Vert \sigma_{h,1} -  \sigma_{h,2} \Vert_{L^p(E)}^p \biggr)^{1/p} 
\biggl( \sum_{E \in \Omega_h} \vert  \vb_h \vert_{[W^{1}_q(\Omega_h)]^2}^q \biggr)^{1/q} 
\biggl( \sum_{E \in \Omega_h}  \vert \ww_h \vert_{1,E}^2  \biggr)^{1/2} 
\\
&\lesssim
\clip 
\Vert \sigma_{h,1} -  \sigma_{h,2} \Vert_{L^p(\Omega)}
\vert  \vb_h \vert_{[W^{1}_q(\Omega_h)]^2}
\vert \ww_h \vert_{1,\Omega} \,.
\end{aligned}
\]
Then, since $p \geq 2$, from the Sobolev embedding \eqref{eq:sobemb} we obtain
\begin{equation}
\label{eq:lmlip2}
\eta_1 
\lesssim
\clip
\vert \sigma_{h,1} -  \sigma_{h,2} \vert_{1,\Omega}
\vert  \vb_h \vert_{[W^{1}_q(\Omega_h)]^2}
\vert \ww_h \vert_{1,\Omega} \,.
\end{equation}
For what concerns the term $\eta_2$ in \eqref{eq:lmlip1}, using similar argument to that used for $\eta_1$, we have
\begin{equation}
\label{eq:lmlip3}
\begin{aligned}
\eta_2 
&\leq 
\sum_{E \in \Omega_h}  \int_{E} \bigl \vert \nu\bigl(\P0 \sigma_{h,1} \bigr)- \nu\bigl(\P0 \sigma_{h,2}  \bigr) \bigr \vert
\,  \bigl\vert \PP0\epseps (\vb_h) \bigr \vert 
\bigl\vert \PP0\epseps (\ww_h) \bigr \vert \, {\rm d}E 
\\
&\lesssim \clip
\sum_{E \in \Omega_h} 
 \Vert \P0 (\sigma_{h,1} -  \sigma_{h,2}) \Vert_{L^p(E)} \,
\vert   \vb_h \vert_{[W^{1}_q(\Omega_h)]^2} \, 
\vert  \ww_h \vert_{1,E}
\\
&\lesssim
\clip
\vert \sigma_{h,1} -  \sigma_{h,2} \vert_{1,\Omega}
\vert  \vb_h \vert_{[W^{1}_q(\Omega_h)]^2}
\vert \ww_h \vert_{1,\Omega} \,.
\end{aligned}
\end{equation}
The thesis follows combining \eqref{eq:lmlip2} and \eqref{eq:lmlip3} in \eqref{eq:lmlip1}.
\end{proof}

We now state the following result concerning the uniqueness of the discrete solution.

\begin{proposition}
\label{prp:uni}
Under the data assumptions \cfan{(A0)}, assume moreover that there exist two positive constants 
$\widehat{C}_{\rm data}$ and $\widehat{C}_{\rm sol}$ depending on the domain $\Omega$, on $k$, on the constants $\kappa_*$, $\kappa^*$ and $\nu_*$ in the data assumptions \cfan{(A0)}
and on the shape regularity constant in  assumption \cfan{(A1)} s.t.
\begin{itemize}
\item  [i)] the data of the problem satisfies
\begin{equation}
\label{eq:cdatah}
\cdatah^2 \left( \vert \ff_h \vert_{\ZDG*}^2  + 
\vert g_h \vert_{-1, \Omega}^2 +
\Vert \tdh \Vert_{1/2, \partial \Omega}^2 \right) < 1 \,,
\end{equation}

\item [ii)] Problem \eqref{eq:pb vem} admits a solution $(\uu_h, p_h, \tinth)$ with $\uu_h \in [W^{1}_q(\Omega_h)]^2$ where $q > 2$ and the following bound holds
\begin{equation}
\label{eq:csolh}
\csolh \, \clip \, \vert \uu_h \vert_{[W^{1}_q(\Omega_h)]^2} < 1 \,.
\end{equation}
\end{itemize}
Then, referring to \eqref{eq:bound disc} and \eqref{eq:cclip}, if
\begin{equation}
\label{eq:uni}
K := \frac{\cconh \cesth}{\alpha_* \nu_* \cdatah} 
\left( 1 + \frac{\cclip}{\beta_* \kappa_* \csolh}
\right) < 1
\end{equation}
this solution is unique.
\end{proposition}

\begin{proof}
We limit to sketch the proof since it follows the guidelines of the proof of \cite[Proposition 2.3]{Bernardi:2015}.
Let $(\uu_{h,i}, \theta_{h,i}) \in \ZDG \times \WDDG$ with $i=1,2$ two solutions of Problem \eqref{eq:pb vem ker} and let
\[
\ddd_h := \uu_{h,1} - \uu_{h,2} \in \ZDG \,,
\qquad
\zeta_h := \theta_{h,1} - \theta_{h,2} \in \WDZG \,.
\]
Then from \eqref{eq:pb vem ker}, employing the skew-symmetry of the convective forms $\cvh$ and $\csh$, we infer
\begin{equation}
\label{eq:prpuni1}
\ash(\zeta_h, \zeta_h) = - \csh(\ddd_h; \theta_{h,1}, \zeta_h)   \,,
\end{equation}
and
\begin{equation}
\label{eq:prpuni2}
\avh(\theta_{h,2} , \, \ddd_h, \ddd_h)  = - \cvh(\ddd_h; \, \uu_{h,1}, \ddd_h)
-\left(\avh(\theta_{h,1}, \, \uu_{h,1}, \ddd_h)  
- \avh(\theta_{h,2}, \, \uu_{h,1}, \ddd_h) \right)
  \,.
\end{equation}
Recalling properties \textbf{(P7)}, from \eqref{eq:prpuni1} we derive
\[
\beta_* \kappa_* \vert \zeta_h \vert_{1, \Omega}^2 
\leq 
\cconh \vert \ddd_h \vert_{1, \Omega} 
\Vert \theta_{h,1} \Vert_{1, \Omega} 
\vert \zeta_h \vert_{1, \Omega} \,.
\]
Thus recalling that $(\uu_{h,1}, \theta_{h,1})$ satisfies bound \eqref{eq:bound disc}, from \eqref{eq:cdatah}
\begin{equation}
\label{eq:prpuni3}
\vert \zeta_h \vert_{1, \Omega} \leq 
\frac{\cconh \cesth}{\cdatah \beta_* \kappa_*} 
\vert \ddd_h \vert_{1, \Omega}  \,.
\end{equation}
In a similar way, employing properties \textbf{(P7)} and Lemma \ref{lm:lipshitz}, from \eqref{eq:prpuni2} we infer
\[
\alpha_* \nu_* \vert \ddd_h \vert_{1, \Omega}^2 
\leq 
\cconh \vert \uu_{h,1} \vert_{1, \Omega} 
\vert \ddd_h \vert_{1, \Omega}^2 
+  \cclip \clip 
\vert \uu_{h,1} \vert_{[W^1_q(\Omega_h)]^2} 
\vert \zeta_h \vert_{1, \Omega} 
\vert \ddd_h \vert_{1, \Omega}\,.
\]
Using again \eqref{eq:bound disc} and \eqref{eq:cdatah}, 
from \eqref{eq:csolh} and \eqref{eq:prpuni3} we derive
\[
\alpha_* \nu_* \vert \ddd_h \vert_{1, \Omega}^2 \leq 
\frac{\cconh \cesth}{\cdatah } 
 \vert \ddd_h \vert_{1, \Omega}^2 +
\frac{\cclip}{\csolh } 
\frac{\cconh \cesth}{\cdatah \beta_* \kappa_*} 
\vert \ddd_h \vert_{1, \Omega} ^2
\,,
\]
that is
\[
\vert \ddd_h \vert_{1, \Omega}^2 \leq  K  \, \vert \ddd_h \vert_{1, \Omega}^2 \,.
\]
Now the proof follows from \eqref{eq:uni}.
\end{proof}

\section{Convergence results: \textit{a priori} analysis}
\label{sec:conv}

In present section we derive the rate of convergence for the proposed VEM scheme \eqref{eq:pb vem}.
We preliminary recall some optimal approximation properties for the discrete temperatures spaces 
$\WDZG$, $\WDDG$ and the velocities space $\VDG$ that can be found in \cite{Sutton:2018,BV:2022} and \cite{BLV:2018} respectively.

\begin{lemma}[Approximation property of $\WDZG$, $\WDDG$]
\label{lm:int WDG}
Under the assumption \cfan{(A1)} for any $\tau \in \Wz \cap H^{\reg+1}(\Omega_h)$ (resp. $\Wd \cap H^{\reg+1}(\Omega_h)$) there exists $\tau_{I} \in \WDZG$ (resp. $\WDDG$) such that for all $E \in \Omega_h$ it holds
\[
\|\tau - \tau_{I}\|_{0,E} + h_E |\tau- \tau_{I}|_{1, E} \lesssim h_E^{\reg+1} |\tau|_{\reg+1,E} 
\]
where $0 < \reg \leq k$.
\end{lemma} 

\begin{lemma}[Approximation property of $\VDG$]
\label{lm:int VDG}
Under the assumption \cfan{(A1)} for any $\vb \in \VCG \cap [H^{\reg+1}(\Omega_h)]^2$ there exists $\vbi \in \VDG$ such that for all $E \in \Omega_h$ it holds
\[
\|\vb - \vbi\|_{0,E} + h_E |\vb- \vbi|_{1, E} \lesssim h_E^{\reg+1} |\vb|_{\reg+1,E} 
\]
where $0 < \reg \leq k$.
Furthermore if $\vb \in \ZCG$ it holds $\vbi \in \ZDG$.
\end{lemma} 


In order to derive the error estimates, we state the following assumption on the solutions and on the data of Problem \eqref{eq:pb primale}.

\smallskip\noindent
\textbf{(A2) Regularity assumptions:} 
\begin{itemize}
\item the solution $(\uu, p, \theta)$ satisfy
$\uu \in [H^{\reg+1}(\Omega_h)]^2$,
$p \in [H^{\reg}(\Omega_h)]$,
$\theta \in [H^{\reg+1}(\Omega_h)]$,

\item the viscosity $\nu$, the conductivity $\kappa$ satisfy 
$\nu(\theta) \in W^{\reg}_\infty(\Omega_h)$,
$\kappa \in W^{\reg}_\infty(\Omega_h)$,

\item the external loads $\ff$ and $g$ satisfy 
$\ff \in [H^{\reg+1}(\Omega_h)]^2$,
$g \in [H^{\reg+1}(\Omega_h)]$.

\end{itemize}
for some $0 < \reg \leq k$.

In the following $\mathcal{R}(\cdot)$ will denote a generic constant depending on the norms of $\uu$, $p$, $\theta$, $\nu$, $\kappa$, $\ff$ and $g$ in the aforementioned spaces.

\begin{proposition}
\label{prp:conv}
Under the assumptions \cfan{(A0)}, \cfan{(A1)} and under the assumptions of Theorem \ref{thm:uni} and Proposition \ref{prp:uni}, let $(\uu, p, \theta)$ be the solution of \eqref{eq:pb variazionale} and $(\uu_h, p_h, \theta_h)$ be the solution of \eqref{eq:pb vem}.
Assuming moreover \cfan{(A2)} the following error estimates hold:
\[
\begin{gathered}
\vert \uu_h - \uu\vert_{1, \Omega} + 
\Vert \theta - \theta_h \Vert_{1, \Omega} \leq \mathcal{R}(\uu, \theta, \nu, \kappa) \, h^\reg \,,
\\
\Vert p - p_h\Vert_{0, \Omega}  \leq \mathcal{R}(\uu, p, \theta, \nu, \kappa) \, h^\reg \,.
\end{gathered}
\]

\end{proposition}

\begin{proof}
We preliminary observe that under the assumptions of Theorem \ref{thm:uni} (resp. Proposition \ref{prp:uni}) and recalling \eqref{eq:bound cont} (resp. \eqref{eq:bound disc}) the following hold
\begin{equation}
\label{eq:bounds}
\vert \uu \vert_{1, \Omega}^2 + \Vert \theta  \Vert_{1, \Omega}^2 \leq 
\frac{\cest^2}{\cdata^2}\,,
\qquad \qquad
\vert \uu_h \vert_{1, \Omega}^2 + \Vert \theta_h  \Vert_{1, \Omega}^2 \leq 
\frac{\cesth^2}{\cdatah^2} \,.
\end{equation}
The proof follows the following steps. 

\noindent
\textbf{Step 1. Interpolation error estimates.}

\noindent
Let us introduce the following error quantities
\[
\begin{aligned}
\eei &:= \uu - \uui\,,
\qquad  &
\si &:= \theta - \theta_{I}\,,
\qquad  &
\rho_{I} &:= p - p_{I}\,,
\\
\ee_h &:= \uui - \uu_h\,,
\qquad  &
\sigma_h &:= \theta_{I} - \theta_h\,,
\qquad  &
\rho_h &:= p_{I} - p_h\,,
\end{aligned}
\]
where $\theta_{I}$ and $\uui$ are the interpolant function of $\theta$ and $\uu$ of Lemma \ref{lm:int WDG} and Lemma \ref{lm:int VDG} respectively and $p_{I} \in \QDG$ is defined elementwise by $p_{I}|_E = \Pi^{0,E}_{k-1} p$ for any $E \in \Omega_h$ (cf. \eqref{eq:P0_k^E}). Therefore from Lemma \ref{lm:int VDG}, Lemma \ref{lm:int WDG} and Bramble-Hilbert Lemma \eqref{eq:bramble}, we have
\begin{equation}
\label{eq:int-errors}
\vert \eei\vert_{1 ,\Omega} \leq \mathcal{R}(\uu) \, h^\reg \,,
\qquad
\Vert \si\Vert_{1 ,\Omega} \leq \mathcal{R}(\theta) \, h^\reg \,,
\qquad
\Vert \rho_{I}\Vert_{1 ,\Omega} \leq \mathcal{R}(p) \, h^\reg \,.
\end{equation}
Notice now that since $\uu \in \ZCG$, from Lemma \ref{lm:int VDG} it holds that $\ee_h \in \ZDG$.
Whereas from Lemma \ref{lm:int WDG},  $\sigma_h \in \WDZG$.

\noindent
\textbf{Step 2. Error equation for the temperature.}

\noindent
Employing Property \textbf{(P7)} and manipulating the second equation in \eqref{eq:pb variazionale ker} and \eqref{eq:pb vem ker} we obtain
\begin{equation}
\label{eq:conv1}
\begin{aligned}
\beta_* \kappa_* \vert \sigma_h \vert_{1, \Omega}^2 
&\leq
\ash(\sigma_h, \sigma_h) =
\ash(\theta_{I}, \sigma_h) - 
\ash(\theta_h, \sigma_h)
\\
& =
\left( 
\ash(\theta_{I}, \sigma_h) - \as(\theta, \sigma_h)
\right)
+
\left( 
\as(\theta, \sigma_h) - \ash(\theta_h, \sigma_h)
\right)
\\
& =
\left( 
\ash(\theta_{I}, \sigma_h) - \as(\theta, \sigma_h)
\right) 
+
\left( 
\csh(\uu_h; \, \theta_h, \sigma_h) - \csskew(\uu; \, \theta, \sigma_h)
\right)
+
(g - g_h, \sigma_h) 
\\
& =: \eta_{a, T} + \eta_{c, T} + \eta_g \,.
\end{aligned}
\end{equation}
We now estimate the three terms above.

\noindent
$\bullet$  estimate of $\eta_{a, T}$: we split $\eta_{a, T}$ into local contributions $\eta_{a, T}^E$. Recalling definitions \eqref{eq:forma as} and \eqref{eq:forma ashl} and the property of the $L^2$-projection we have
\[
\begin{aligned}
\eta_{a, T}^E &:= 
\int_{E} \kappa \, \Pp0\nabla \theta_{I} \cdot \Pp0 \nabla \sigma_h \, {\rm d}E 
+ \ssl(\theta_{I} , \, \sigma_h)  -
\int_{E} \kappa \, \nabla \theta \cdot \nabla \sigma_h \, {\rm d}E 
\\
& =
\int_{E} \kappa \, (\Pp0\nabla \theta_{I} - \nabla \theta)\cdot \Pp0 \nabla \sigma_h \, {\rm d}E 
- \int_{E} \kappa \, \nabla \theta \cdot (I - \Pp0)\nabla \sigma_h \, {\rm d}E 
+ \ssl(\theta_{I} , \, \sigma_h) 
\\
& =
\int_{E} \kappa \, (\Pp0\nabla \theta_{I} - \nabla \theta)\cdot \Pp0 \nabla \sigma_h \, {\rm d}E 
- \int_{E} (I - \Pp0) (\kappa \, \nabla \theta) \cdot \nabla \sigma_h \, {\rm d}E 
+ \ssl(\theta_{I} , \, \sigma_h) \,.
\end{aligned}
\]
Then from Property \textbf{(P0)}, \eqref{eq:ssstab} and \eqref{eq:stab}, employing Bramble-Hilbert Lemma \eqref{eq:bramble}, we obtain 
\[
\begin{aligned}
\eta_{a, T}^E & \leq
\left( 
\kappa^* \Vert \Pp0\nabla \theta_{I} - \nabla \theta\Vert_{0,E} + 
\Vert (I - \Pp0) (\kappa \, \nabla \theta) \Vert_{0,E} +
\kappa^* \vert (I - \P0) \theta_{I}\vert_{1, E}
\right) \vert \sigma_h \vert_{1,E}
\\
& \leq
\kappa^* 
\left( 
\vert \si\vert_{1,E} + 
\vert (I - \P0)\theta\vert_{1,E} +
\frac{1}{\kappa^*} \Vert (I - \Pp0) (\kappa \, \nabla \theta) \Vert_{0,E} +
\right) \vert \sigma_h \vert_{1,E}
\\
& \lesssim
\kappa^* h_E^{\reg}
\left(  
\vert \si\vert_{1,E} + 
\vert \theta\vert_{\reg+1,E} +
\frac{\Vert \kappa\Vert_{W^\reg_\infty}(E)}{\kappa^*}
\Vert \theta\Vert_{\reg+1,E}
\right) \vert \sigma_h \vert_{1,E} \,.
\end{aligned}
\]
Therefore summing the local contributions and employing \eqref{eq:int-errors} we obtain
\begin{equation}
\label{eq:conv2}
\eta_{a,T} \leq 
\mathcal{R}(\theta, \kappa) \, h^{\reg} \,
\vert \sigma_h \vert_{1,\Omega} \,.
\end{equation}

\noindent
$\bullet$  estimate of $\eta_{c, T}$: employing the skew-symmetry of $\csh(\cdot; \cdot, \cdot)$ (cf. \eqref{eq:forma cshlskew}) we have
\begin{equation}
\label{eq:conv3a}
\begin{aligned}
\eta_{c,T} &= 
\csh(\uu_h; \, \theta_h, \sigma_h) - 
\csh(\uu; \, \theta, \sigma_h)  + 
\csh(\uu; \, \theta, \sigma_h) - 
\csskew(\uu; \, \theta, \sigma_h) 
\\
&= 
- \csh(\uu_h; \, \si, \sigma_h)
- \csh(\uu - \uu_h; \, \theta, \sigma_h)
+ 
 \csh(\uu; \, \theta, \sigma_h) - 
\csskew(\uu; \, \theta, \sigma_h)  \,.
\end{aligned}
\end{equation}
From Property \textbf{(P7)}, bounds \eqref{eq:bounds} and \eqref{eq:int-errors} we have
\begin{equation}
\label{eq:conv3b}
\begin{aligned}
\csh(\uu_h; \, \si, \sigma_h) &+
\csh(\uu - \uu_h; \, \theta, \sigma_h) 
\leq \cconh \left(
\vert \uu_h\vert_{1, \Omega} \Vert \si\Vert_{1, \Omega} + 
\vert \uu - \uu_h\vert_{1, \Omega} \Vert \theta\Vert_{1, \Omega} 
\right) \vert \sigma_h\vert_{1, \Omega}
\\
& \leq  
\frac{\cconh \cesth}{\cdatah} 
 \Vert \si\Vert_{1, \Omega}  \vert \sigma_h\vert_{1, \Omega} + 
\frac{\cconh \cest}{\cdata} 
\left(\vert \eei\vert_{1, \Omega} + 
\vert \ee_h\vert_{1, \Omega}   
\right) \vert \sigma_h\vert_{1, \Omega}
\\
& \leq  
\frac{\cconh \cest}{\cdata} \vert\ee_h\vert_{1, \Omega}   
\vert \sigma_h\vert_{1, \Omega} +
\mathcal{R}(\uu, \theta)\, h^\reg \, \vert \sigma_h\vert_{1, \Omega} \,.
\end{aligned}
\end{equation}
Furthermore employing \cite[Lemma 4.3]{BLV:2018}
\begin{equation}
\label{eq:conv3c}
\csh(\uu; \, \theta, \sigma_h) - 
\csskew(\uu; \, \theta, \sigma_h)  \leq \mathcal{R}(\uu, \theta) \, h^{\reg} \,
\vert \sigma_h \vert_{1,\Omega} \,.
\end{equation}
Therefore collecting in \eqref{eq:conv3a}, bounds \eqref{eq:conv3b} and \eqref{eq:conv3c} we obtain
\begin{equation}
\label{eq:conv3}
\eta_{c,T} \leq 
\frac{\cconh \cest}{\cdata} \vert\ee_h\vert_{1, \Omega}   
\vert \sigma_h\vert_{1, \Omega} +
\mathcal{R}(\uu, \theta) \, h^{\reg} \,
\vert \sigma_h \vert_{1,\Omega} \,.
\end{equation}

\noindent
$\bullet$  estimate of $\eta_{g}$: employing the definition of $L^2$-projection \eqref{eq:P0_k^E} and  Bramble-Hilbert Lemma \eqref{eq:bramble} we infer
\begin{equation}
\label{eq:conv4}
\eta_g \leq \mathcal{R}(g) \, h^{\reg +2} \, \vert \sigma_h \vert_{1,\Omega} \,.
\end{equation}
Finally combining \eqref{eq:conv2}, \eqref{eq:conv3} and \eqref{eq:conv4} in \eqref{eq:conv1} we get

\begin{equation}
\label{eq:convtheta}
\vert {\sigma_h} \vert_{1, \Omega} \leq 
\frac{\cconh \cest}{\cdata \beta_* \kappa_*} \vert\ee_h\vert_{1, \Omega}   +
\mathcal{R}(\uu, \theta, \kappa) \, h^{\reg} + 
\mathcal{R}(g) \, h^{\reg +2} \,.
\end{equation}

\noindent
\textbf{Step 3. Error equation for the velocity.}

\noindent
We now analyse the velocity error equation. Combining the first equation in \eqref{eq:pb variazionale ker} and \eqref{eq:pb vem ker} together with Property \textbf{(P7)}, we have
\begin{equation}
\label{eq:conv5}
\begin{aligned}
&\alpha_* \nu_* \vert \ee_h \vert_{1, \Omega}^2 
\leq
\avh(\theta_h; \, \ee_h, \ee_h) =
\avh(\theta_h; \, \uui, \ee_h) - 
\avh(\theta_h; \, \uu_h, \ee_h)
\\
& =
\left( 
\avh(\theta_h; \, \uui, \ee_h) - \av(\theta; \,  \uu, \ee_h)
\right)
+
\left( 
\av(\theta; \,  \uu, \ee_h) - \avh(\theta_h; \, \uu_h, \ee_h)
\right)
\\
& =
\left( 
\avh(\theta_h; \, \uui, \ee_h) - \av(\theta; \,  \uu, \ee_h)
\right)
+
\left( 
\cvh(\uu_h; \, \uu_h, \ee_h) - \cvskew(\uu; \, \uu, \ee_h)
\right)
+
(\ff - \ff_h, \ee_h) 
\\
& =: \eta_{a, V} + \eta_{c, V} + \eta_{\ff} \,.
\end{aligned}
\end{equation}
We now estimate each term above.

\noindent
$\bullet$  estimate of $\eta_{a, V}$: we split the term into local contributions $\eta_{a, V}^E $. Recalling definitions \eqref{eq:forma av} and \eqref{eq:forma avhl} we have
\begin{equation}
\label{eq:conv6a}
\begin{aligned}
\eta_{a, V}^E &:= 
\int_{E} \nu(\P0 \theta_h) \, \PP0\gr \uui \cdot \PP0 \gr \ee_h \, {\rm d}E 
+ \svl(\theta_h; \, \uui , \, \ee_h)  -
\int_{E} \nu(\theta) \, \gr \uu \cdot \gr \ee_h \, {\rm d}E
\\
& =
\int_{E} \nu(\P0 \theta_h) \, (\PP0\gr \uui - \gr \uu) \cdot \PP0 \gr \ee_h \, {\rm d}E 
+ \svl(\theta_h; \, \uui , \, \ee_h) +
\\
& \quad
- \int_{E} \nu(\theta) \, \gr \uu \cdot (I - \PP0)\gr \ee_h \, {\rm d}E 
+ 
\int_{E} (\nu(\P0 \theta_h) - \nu(\theta)) \, \gr \uu \cdot \PP0 \gr \ee_h \, {\rm d}E
\\
& =: \zeta_1^E + \zeta_2^E + \zeta_3^E + \zeta_4^E \,.
\end{aligned}
\end{equation}
Recalling Property \textbf{(P0)}, the stability bounds \eqref{eq:svstab} and \eqref{eq:stab}, and the definition of $L^2$-projection \eqref{eq:P0_k^E}, the first three sub-terms can be bounded as follows
\[
\begin{aligned}
\sum_{i=1}^3\zeta_i^E &\leq 
\bigl(
\nu^* \Vert \PP0\gr \uui - \gr \uu\Vert_{0,E} +
\Vert (I - \PP0)(\nu(\theta)\gr \uu)\Vert_{0,E} +
\nu^* \vert ( I - \P0) \uui\vert_{1,E}
\bigr)
\vert \ee_h  \vert_{1,E}
\\
& \leq
\nu^* 
\left( 
\vert  \eei\vert_{1,E} + 
{\vert (I - \P0) \uu\vert_{1,E}} +
\frac{1}{\nu^*} \Vert (I - \P0) (\nu(\theta) \, \gr \uu) \Vert_{0,E} 
\right) \vert \ee_h \vert_{1,E}
\\
& \lesssim
\nu^* h_E^{\reg}
\left(  
\vert  \eei\vert_{1,E} + 
\vert \uu\vert_{\reg+1,E} +
\frac{\Vert \nu(\theta)\Vert_{W^\reg_\infty}(E)}{\nu^*}
\Vert \uu\Vert_{\reg+1,E}
\right) \vert \ee_h \vert_{1,E} \,.
\end{aligned}
\]
Therefore
\begin{equation}
\label{eq:conv6b}
\begin{aligned}
\sum_{E \in \Omega_h}\sum_{i=1}^3\zeta_i^E &
\leq 
\mathcal{R}(\uu, \nu) \, h^\reg \, 
\vert \ee_h \vert_{1, \Omega} \,.
\end{aligned}
\end{equation}
For the term $\zeta_4$, employing similar computations to that in \eqref{eq:lmlip3} (recalling that $p$ is such that $1/p + 1/q =1/2)$ we infer
\[
\begin{aligned}
\zeta_4^E &\leq
\clip \Vert \P0 \theta_h - \theta\Vert_{L^p(E)} 
\vert \uu \vert_{[W^1_q(E)]^2} \vert\ee_h \vert_{1,E}
\\
& \leq
\clip \left( 
\Vert \P0 (\theta_h - \theta)\Vert_{L^p(E)} +
\Vert (I - \P0) \theta\Vert_{L^p(E)}
\right)
\vert \uu \vert_{[W^1_q(E)]^2} \vert\ee_h \vert_{1,E} 
\\
& \lesssim
\clip \left( 
\Vert \theta_h - \theta\Vert_{L^p(E)} +
\Vert \theta\Vert_{W^\reg_p(E)} \, h_E^\reg
\right)
\vert \uu \vert_{[W^1_q(E)]^2} \vert\ee_h \vert_{1,E} 
\\
& \lesssim
\clip \left( 
\Vert \sigma_h\Vert_{L^p(E)} +\Vert \si\Vert_{L^p(E)} +
\Vert \theta\Vert_{W^\reg_p(E)} \, h_E^\reg
\right)
\vert \uu \vert_{[W^1_q(E)]^2} \vert\ee_h \vert_{1,E} 
\end{aligned}
\]
where we also used the continuity of the $L^2$-projection w.r.t. the $L^p$-norm and Bramble-Hilbert \eqref{eq:bramble}.
From the H\"older inequality for sequences and the Sobolev embedding \eqref{eq:sobemb}, we infer
\[
\sum_{E \in \Omega_h} \zeta_4^E
\lesssim
\clip \left( 
\vert \sigma_h\vert_{1, \Omega} +\Vert \si\Vert_{1,\Omega} +
\mathcal{R}(\theta) h^\reg
\right)
\vert \uu \vert_{[W^1_q(\Omega_h)]^2} \vert\ee_h \vert_{1,\Omega} \,.
\]
Therefore from \eqref{eq:int-errors} and \eqref{eq:csol}
\begin{equation}
\label{eq:conv6c}
\begin{aligned}
\sum_{E \in \Omega_h} \zeta_4^E
& \leq
\frac{C_{\rm lip}}{\csol} \vert \sigma_h\vert_{1, \Omega} \vert\ee_h \vert_{1,\Omega}  +
\mathcal{R}(\theta) \, h^\reg \,
 \vert\ee_h \vert_{1,\Omega} 
\end{aligned}
\end{equation}
for a suitable positive constant $C_{\rm lip}$.
Collecting \eqref{eq:conv6a} and \eqref{eq:conv6b} in \eqref{eq:conv6c} we obtain
\begin{equation}
\label{eq:conv6}
\begin{aligned}
\eta_{a,V} & \leq 
\frac{C_{\rm lip}}{\csol} \vert \sigma_h \vert_{1, \Omega}
\vert \ee_h\vert_{1, \Omega} +
\mathcal{R}(\uu, \theta, \nu) \, h^\reg \, 
\vert \ee_h\vert_{1, \Omega}   \,.
\end{aligned}
\end{equation}

\noindent
$\bullet$  estimate of $\eta_{c,V}$:  direct application of \cite[Lemma 4.3 and Lemma 4.4]{BLV:2018} yields
\begin{equation}
\label{eq:conv7}
\eta_{c,V} \leq 
\frac{\cconh \cesth}{\cdatah} \vert\ee_h\vert_{1, \Omega}^2 +
\mathcal{R}(\uu) \, h^{\reg} \,
\vert \ee_h \vert_{1,\Omega} \,.
\end{equation}

\noindent
$\bullet$  estimate of $\eta_{\ff}$: using the same argument used above for $\eta_g$ we obtain
\begin{equation}
\label{eq:conv8}
\eta_{\ff} \leq \mathcal{R}(\ff) \, h^{\reg +2} \, \vert \ee_h \vert_{1,\Omega} \,.
\end{equation}
Collecting \eqref{eq:conv6}, \eqref{eq:conv7} and \eqref{eq:conv8} in the error equation \eqref{eq:conv5} we obtain
\begin{equation}
\label{eq:convu1}
\alpha_* \nu_* \vert \eei\vert_{1, \Omega} \leq
\mathcal{R}(\uu, \theta, \nu) h^\reg +
\mathcal{R}(\ff) h^{\reg+2} + 
\frac{\cconh \cesth}{\cdatah} \vert\ee_h\vert_{1, \Omega} +
\frac{C_{\rm lip}}{\csol} \vert \sigma_h\vert_{1, \Omega} \,.
\end{equation}

\noindent
\textbf{Step 4. Error estimates for the velocity and the temperature}

\noindent
Now employing the estimate \eqref{eq:convtheta} in \eqref{eq:convu1} we infer
\[
\left(\alpha_* \nu_* -
\frac{\cconh \cesth}{\cdatah} - 
\frac{\cconh \cest}{\cdata}  
\frac{C_{\rm lip}}{\csol \beta_* \kappa_*}
\right)
\vert \eei\vert_{1, \Omega} \leq
\mathcal{R}(\uu, \theta, \nu, \kappa) h^\reg +
\mathcal{R}(\ff, g) h^{\reg+2} \,.
\]
Therefore assuming that $\nu_*$ and $\kappa_*$ large enough 
\begin{equation}
\label{eq:convu2}
\vert \eei\vert_{1, \Omega} \leq
\mathcal{R}(\uu, \theta, \nu, \kappa) h^\reg +
\mathcal{R}(\ff, g) h^{\reg+2} \,.
\end{equation}
We finally collect the estimates \eqref{eq:convu2}, \eqref{eq:convtheta} and the error bounds \eqref{eq:int-errors}  obtaining
\begin{equation}
\label{eq:conv-theta-u}
\vert \uu - \uu_h \vert_{1, \Omega} +
\Vert \theta - \theta_h \Vert_{1, \Omega}
\leq 
\mathcal{F}(\uu, \theta, \nu, \kappa) h^\reg +
\mathcal{F}(\ff, g) h^{\reg+2}  \,.
\end{equation}

\noindent
\textbf{Step 5. Error estimate for the pressure.}

\noindent
In the last step we briefly study the convergence error for the pressures. 
From the first equation in \eqref{eq:pb variazionale} and \eqref{eq:pb vem}, simple computations yield
\begin{equation}
\label{eq:conv10}
\begin{aligned}
b(\vb_h, \rho_h) &= 
b(\vb_h, p_{I}) - b(\vb_h, p_h) =
b(\vb_h, p) - b(\vb_h, p_h) -  b(\vb_h, \rho_{I})
\\
& =
\left( \avh(\theta_h; \, \uu_h, \vb_h) - \av(\theta; \, \uu, \vb_h) \right) + 
\left( \cvh(\uu_h; \, \uu_h, \vb_h) - \cvskew(\uu; \, \uu, \vb_h) \right) +
\\
& \qquad
+ (\ff- \ff_h, \vb_h) -  b(\vb_h, \rho_{I})
\\
& := \gamma_{a, V} + \gamma_{c, V} + \gamma_{\ff} + \gamma_b
 \,.
\end{aligned}
\end{equation}

\noindent
$\bullet$ estimate of $\gamma_{a, V}$: the estimate of $\gamma_{a, V}$ follows the same steps for the bound of $\eta_{a, V}$
\[
\begin{aligned}
\gamma_{a, V} &\leq 
\nu^* \vert \uu - \uu_h\vert_{1, \Omega} \vert \vb_h\vert_{1, \Omega} +
\frac{C_{\rm lip}}{\csol} \Vert \theta - \theta_h \Vert_{1, \Omega}
\vert \vb_h\vert_{1, \Omega} +
\mathcal{R}(\uu, \theta, \nu) \, h^\reg \, 
\vert \vb_h\vert_{1, \Omega} 
\end{aligned}
\]
therefore from \eqref{eq:conv-theta-u} we have
\begin{equation}
\label{eq:conv11}
\begin{aligned}
\gamma_{a, V} &\leq 
\left(\mathcal{R}(\uu, \theta, \nu, \kappa) \, h^\reg +
\mathcal{R}(\ff, g) \, h^{\reg+2}\right)
\vert \vb_h\vert_{1, \Omega} \,.
\end{aligned}
\end{equation}

\noindent
$\bullet$ estimate of $\gamma_{c, V}$: simple computations yield
\begin{equation}
\label{eq:conv12a}
\begin{aligned}
\gamma_{c, V} &=  
\left(\cvh(\uu_h; \, \uu_h, \vb_h) -
\cvh(\uu; \, \uu, \vb_h) \right) +
\left(\cvh(\uu; \, \uu, \vb_h) -
\cvskew(\uu; \, \uu, \vb_h)\right)
\\
&=
\left(\cvh(\uu_h - \uu; \, \uu_h, \vb_h) +
\cvh(\uu; \, \uu_h - \uu, \vb_h) \right) +
\left(\cvh(\uu; \, \uu, \vb_h) -
\cvskew(\uu; \, \uu, \vb_h)\right) \,.
\end{aligned}
\end{equation}
Employing Property \textbf{(P7)} and bounds \eqref{eq:bounds} and estimate \eqref{eq:conv-theta-u}, we have
\begin{equation}
\label{eq:conv12b}
\begin{aligned}
\cvh(\uu_h - \uu; \, \uu_h, \vb_h) +
\cvh(\uu; \, \uu_h - \uu, \vb_h) 
& \leq
\cconh \left( \vert \uu \vert_{1, \Omega} + 
\vert \uu_h \vert_{1, \Omega}
\right) \vert \uu - \uu_h \vert_{1, \Omega} \,
\vert \vb_h\vert_{1, \Omega}
\\
& \leq
\cconh \left( \frac{\cest}{\cdata} + \frac{\cesth}{\cdatah} \right)
\vert \uu - \uu_h \vert_{1, \Omega} \,
\vert \vb_h\vert_{1, \Omega}
\\
& \leq
\left(\mathcal{R}(\uu, \theta, \nu, \kappa) \, h^\reg +
\mathcal{R}(\ff, g) \, h^{\reg+2}\right)
\vert \vb_h\vert_{1, \Omega} \,.
\end{aligned}
\end{equation}
Whereas applying \cite[Lemma 4.3]{BLV:2018} we have
\begin{equation}
\label{eq:conv12c}
\begin{aligned}
\cvh(\uu; \, \uu, \vb_h) -
\cvskew(\uu; \, \uu, \vb_h)
& \leq
\mathcal{R}(\uu) \, h^\reg \,
\vert \vb_h\vert_{1, \Omega} \,.
\end{aligned}
\end{equation}
Therefore \eqref{eq:conv12a}, \eqref{eq:conv12b} and \eqref{eq:conv12c} implies
\begin{equation}
\label{eq:conv12}
\begin{aligned}
\gamma_{c, V} \leq \left(\mathcal{R}(\uu, \theta, \nu, \kappa) \, h^\reg +
\mathcal{R}(\ff, g) \, h^{\reg+2}\right)
\vert \vb_h\vert_{1, \Omega} \,.
\end{aligned}
\end{equation}

\noindent
$\bullet$ estimate of $\gamma_{\ff} + \gamma_b$: the term $\gamma_{\ff}$ can be bounded as $\eta_{\ff}$. The term $\gamma_b$ can be bounded, recalling Property \textbf{(P0)}, by 
Bramble-Hilbert Lemma \ref{eq:bramble}. Then we have
\begin{equation}
\label{eq:conv13}
\gamma_{\ff} + \gamma_b \leq 
\left(\mathcal{R}(\ff) \, h^{\reg+2} + 
\mathcal{R}(p) \, h^\reg 
\right) 
\vert \vb_h\vert_{1, \Omega} \,.
\end{equation}
Employing the inf-sup stability \eqref{eq:infsup}, from \eqref{eq:conv11}--\eqref{eq:conv13}, we obtain
\[
\isd \|\rho_h\|_{0,E}  \leq
\sup_{\vb_h \in \VDG} \frac{b(\vb_h, \rho_h)}{\vert \vb_h\vert_{1, \Omega}}
\leq 
\mathcal{F}(\uu, \theta, p, \nu, \kappa) \, h^\reg +
\mathcal{F}(\ff, g) \, h^{\reg+2} \,.
\]
The pressure error estimates follows from the previous bound and \eqref{eq:int-errors}:
\[
\|p - p_h\|_{0,E}  \leq
\mathcal{F}(\uu, \theta, p, \nu, \kappa) \, h^\reg +
\mathcal{F}(\ff, g) \, h^{\reg+2} \,.
\]

\end{proof}

\begin{remark}
\label{rmk:conv}
Note that the proposed VEM scheme \eqref{eq:pb vem ker}  has the following favorable property that extends to the context of coupled problems the convergence result obtained for the Navier--Stokes equation:
the  error components partly decouple.
In fact the velocity and the temperature errors do not
depend directly on the discrete pressures, but only indirectly through the approximation
of the loading and convection terms and such dependence on the full load is
much weaker with respect to standard mixed schemes. 
In some situations the partial
decoupling of the errors induces a positive effect on the velocity/temperature  approximation (see for instance Test 2).
\end{remark}

\section{Numerical results}
\label{sec:num}

In this section we present two numerical experiments to test the practical performance of the proposed VEM scheme \eqref{eq:pb vem} and the possible advantages related to the divergence-free velocity solutions.
In the  tests we consider $k=2$.
\subsection{Fixed point iteration}

We describe the linearization strategy based on a fixed-point iteration adopted to solve the non-linear coupled problem.

\vspace{0.2cm}

\noindent
\texttt{LINEAR FIXED POINT ITERATION}

\noindent
Starting from $(\uu^0_h, \, p^0_h) =(\boldsymbol{0}, 0)$, for $n \geq 0$, until convergence solve
\begin{itemize}
\item \texttt{HEAT equation}
\begin{equation*}
\label{eq:heat}
\left \{
\begin{aligned}
& \text{find $\theta_h^{n+1}\in \WDDG$, such that} \\
&\ash(\theta_h^{n+1}, \sigma_h) + \csh(\uu_h^{n}; \, \theta_h^{n+1}, \sigma_h) = (g_h, \sigma_h) \qquad & &\text{$\forall \sigma_h \in \WDZG$,} 
\end{aligned}
\right.
\end{equation*}

\item \texttt{OSEEN equation}
\begin{equation*}
\label{eq:oseen}
\left\{
\begin{aligned}
&\text{find $(\uu_h^{n+1}, \,  p_h^{n+1}) \in \VV_h \times Q_h$, such that} \\
&\begin{aligned}
\avh(\theta_h^{n+1}; \, \uu_h^{n+1}, \vb_h) + \cvh(\uu_h^{n}; \, \uu_h^{n+1}, \vb_h)  + b(\vb_h, p_h^{n+1}) &= (\ff_h, \vb_h) 
\quad & &\text{$\forall \vb_h \in \VDG$,} \\
b(\uu_h^{n+1}, q_h) &= 0 \quad & & \text{$\forall  q_h \in \QDG$.} 
\end{aligned}
\end{aligned}
\right.
\end{equation*}
\end{itemize}
Notice that at each step of the fixed-point iteration the
solution $\uu_h^n$ is still divergence-free, therefore the linearization procedure does non
affect the divergence-free property of the final discrete solution.
At each iteration we need to solve two linear systems with dimension
\[
\begin{aligned}
\texttt{N\_DoFs HEAT} &= N_V + (k-1)N_E + \frac{k(k-1)}{2} N_P \,,
\\
\texttt{N\_DoFs OSEEN} &= 2\, N_V + 2\,(k-1)N_E + 
\left(\frac{k(3k-1)}{2} \right) N_P + 1 \,,
\end{aligned}
\]
where $N_V$, $N_E$, $N_P$ denotes the number of internal vertices, internal edges and polygons in $\Omega_h$ respectively.
The dimension of the linear system for the \texttt{OSEEN equation} can be significantly reduced considering the so-called reduced spaces (see \cite{BLV:2018}).
We now derive the convergence of the fixed-point iteration.
\begin{proposition}
\label{prp:fixed point}
Under the assumptions \cfan{(A0)}, \cfan{(A1)} and under the assumptions of Proposition \ref{prp:uni}, let $(\uu_h, p_h, \theta_h)$ be the solution of \eqref{eq:pb vem}.
Then the sequence $\{(\uu_h^n, p_h^n, \theta_h^n)\}_n$ generated by the \cfun{LINEAR FIXED POINT ITERATION} satisfies
\[
\lim_{n \to \infty} \vert \uu_h - \uu_h^n\vert_{1, \Omega} = 0
\qquad
\lim_{n \to \infty} \vert \theta_h - \theta_h^n\vert_{1, \Omega} = 0
\qquad
\lim_{n \to \infty} \Vert p_h - p_h^n\Vert_{0, \Omega} = 0 \,.
\]
\end{proposition}

\begin{proof}
We only sketch the proof since it can be derived using similar technique to that in the proof of Proposition \ref{prp:uni}. 
For any $n \geq 0$, let us introduce the error quantities
\[
\ddd_h^n := \uu_h^n - \uu_h\,,
\qquad
\sigma_h^n := \theta_h^n - \theta_h\,,
\qquad
\rho_h^n := p_h^n - p_h \,.
\]
Note that $\ddd_h^n \in \ZDG$ and $\sigma_h^n \in \WDZG$.
Simple computations yield
\[
\ash(\sigma_h^{n+1}, \sigma_h^{n+1}) = - 
\csh(\ddd_h^n; \theta_h, \sigma_h^{n+1}) \,.
\]
Therefore form Property \textbf{(P7)} and bound \eqref{eq:bounds} we obtain.
\begin{equation}
\label{eq:sigmahn}
\beta_* \kappa_* \vert \sigma_h^{n+1}\vert_{1, \Omega}
\leq 
\frac{\cconh \cesth}{\cdatah}  \vert \ddd_h^n\vert_{1, \Omega} \,.
\end{equation}
For the velocity we can derive the following equation
\[
\avh(\theta_h^{n+1} ; \, \ddd_h^{n+1}, \ddd_h^{n+1})  = - \cvh(\ddd_h^n; \, \uu_h, \ddd_h^{n+1})
-\left(\avh(\theta_h^{n+1}; \, \uu_h, \ddd_h^{n+1})  
- \avh(\theta_h; \, \uu_h, \ddd_h^{n+1}) \right) \,.
\]
Then, employing properties \textbf{(P7)} and Lemma \ref{lm:lipshitz} we infer
\[
\alpha_* \nu_* \vert \ddd_h^{n+1} \vert_{1, \Omega}^2 
\leq 
\cconh \vert \uu_h \vert_{1, \Omega} 
\vert \ddd_h^n \vert_{1, \Omega} 
\vert \ddd_h^{n+1} \vert_{1, \Omega} 
+  \cclip \clip 
\vert \uu_h \vert_{[W^1_q(\Omega_h)]^2} 
\vert \sigma_h^{n+1} \vert_{1, \Omega} 
\vert \ddd_h^{n+1} \vert_{1, \Omega}\,.
\]
Using  \eqref{eq:bounds} and \eqref{eq:csolh}, from \eqref{eq:sigmahn} we derive
\[
\alpha_* \nu_* \vert \ddd_h^{n+1} \vert_{1, \Omega} \leq 
\frac{\cconh \cesth}{\cdatah } 
 \vert \ddd_h^n \vert_{1, \Omega} +
\frac{\cclip}{\csolh } 
\frac{\cconh \cesth}{\cdatah \beta_* \kappa_*} 
\vert \ddd_h^n \vert_{1, \Omega}
\,,
\]
that, recalling \eqref{eq:uni}, corresponds to
\[
\vert \ddd_h^{n+1} \vert_{1, \Omega} \leq  K  \, \vert \ddd_h^n \vert_{1, \Omega}\,.
\]
Therefore, since $\uu_h^0 = \boldsymbol{0}$ from \eqref{eq:bounds}, we have
\begin{equation}
\label{eq:thetahn}
\vert \ddd_h^n \vert_{1, \Omega} \leq 
\frac{\cesth}{\cdatah} \, K^n 
\,, 
\qquad \qquad
\vert \sigma_h^n \vert_{1, \Omega} \leq
\frac{ \cconh \cesth^2}{\cdatah^2 \beta_* \kappa_*}  
K^{n-1} \,.
\end{equation}
Concerning the pressure we derive 
\begin{equation}
\label{eq:rhohn1}
\begin{aligned}
b(\vb_h, \rho_h^{n+1}) &=
\avh(\theta_h; \, \uu_h, \vb_h) - 
\avh(\theta_h^{n+1}; \, \uu_h^{n+1}, \vb_h)  +
\\
& \qquad + 
\cvh(\uu_h; \, \uu_h, \vb_h) - 
\cvh(\uu_h^{n}; \, \uu_h^{n+1}, \vb_h) 
\\
&=
\left(\avh(\theta_h; \, \uu_h, \vb_h) - 
\avh(\theta_h^{n+1}; \, \uu_h, \vb_h)  \right)
- \avh(\theta_h^{n+1}; \, \ddd_h^{n+1}, \vb_h)  
\\
& \qquad - 
\cvh(\ddd_h^n; \, \uu_h, \vb_h) - 
\cvh(\uu_h^{n}; \, \ddd_h^{n+1}, \vb_h) \,.
\end{aligned}
\end{equation}
Lemma \ref{lm:lipshitz} implies
\begin{equation}
\label{eq:rhohn2}
\avh(\theta_h; \, \uu_h, \vb_h) - 
\avh(\theta_h^{n+1}; \, \uu_h, \vb_h) \leq
\cclip \clip \vert \sigma_h^{n+1}\vert_{1, \Omega}
\vert \uu_h\vert_{[W^1_q(\Omega)]^2}
\vert \vb_h\vert_{1, \Omega} \,,
\end{equation}
whereas Property \textbf{(P7)} yields
\begin{equation}
\label{eq:rhohn3}
\begin{aligned}
-\avh(\theta_h^{n+1}; \, \ddd_h^{n+1}, \vb_h)  
&\leq 
\alpha^* \nu^* 
\vert \ddd_h^{n+1}\vert_{1, \Omega}  
\vert \vb_h\vert_{1, \Omega}
\\
- \cvh(\ddd_h^n; \, \uu_h, \vb_h) - 
\cvh(\uu_h^{n}; \, \ddd_h^{n+1}, \vb_h) 
& \leq
\cconh
\left(
\vert \ddd_h^{n}\vert_{1, \Omega} 
\vert \uu_h\vert_{1, \Omega}
+
\vert \ddd_h^{n+1}\vert_{1, \Omega} 
\vert \uu_h^{n}\vert_{1, \Omega} 
\right)
\vert \vb_h\vert_{1, \Omega} \,.
\end{aligned}
\end{equation}
Then recalling \eqref{eq:csolh} and \eqref{eq:bounds} from \eqref{eq:rhohn1}, \eqref{eq:rhohn2} and \eqref{eq:rhohn3}, for any $\vb_h \in \VDG$ we have
\[
\begin{aligned}
\frac{b(\vb_h, \rho_h^{n+1})}{{\vert\vb_h\vert_{1,\Omega}}} &\leq
\frac{\cclip}{\csolh} \vert \sigma_h^{n+1}\vert_{1, \Omega}
+
\left(\alpha^* \nu^*  + \cconh \frac{\cesth}{\cdatah} \right) 
\vert \ddd_h^{n+1} \vert_{1, \Omega}
+
\cconh \frac{\cesth}{\cdatah} \vert \ddd_h^{n} \vert_{1, \Omega} \,.
\end{aligned}
\]
Therefore from the inf-sup stability \eqref{eq:infsup} and \eqref{eq:thetahn} we have
\[
\isd \Vert \rho_h^{n+1}\Vert_{0, \Omega}
\leq
\cconh \frac{\cesth^2}{\cdatah^2} \left( 
1 + \frac{\cclip}{\csolh \beta_* \kappa_*}
\right)
K^n + O(K^{n+1}) \,.
\]
Now the {thesis} follows from \eqref{eq:uni}.
\end{proof}

\paragraph{\textbf{Test 1. Error convergence}}
In this test we examine the practical performance and the convergence properties of the proposed scheme \eqref{eq:pb vem} in the light of Proposition \ref{prp:conv}.

\noindent
In order to compute the VEM error between the exact solution $\uu_{\rm ex}$ and the VEM
solution $\uu_h$ we consider the computable error quantities 
\[
\begin{aligned}
\texttt{err}(\uu_h, H^1) &:=
\texttt{sqrt}\left(\sum_{E \in \Omega_h} \|\Gr \uu_{\rm ex} - \Pi^{0,E}_{1} \Gr \uu_h\|^2_{0,E}\right) /\|\Gr \uu_{\rm ex}\|_{0}  \,,
\\
\texttt{err}(\uu_h, L^2) &:=
\texttt{sqrt}\left(\sum_{E \in \Omega_h} \|\uu_{\rm ex} - \Pi^{0,E}_{2} \uu_h\|_{0,E}\right) /\|\uu_{\rm ex}\|_{0}
\,.
\end{aligned}
\]
Analogous error quantities are considered to measure the error between the exact solution $\theta_{\rm ex}$ and the VEM solution $\theta_h$.
For the pressure error we take $\texttt{err}(p_h, L^2) := \|p - p_h\|_{0, \Omega}$.

In the present test we consider Problem \eqref{eq:pb primale} on the unit square 
$\Omega = (0, 1)^2$, the viscosity is $\nu(\theta)=2 + \theta^2$, the conductivity is $\kappa(x, y)= 1$, the load terms $\ff$ and $g$ and the Dirichlet boundary conditions are chosen in accordance with the analytical solution
\[
\begin{aligned}
\uu_{\rm ex}(x,y) &= e^x 
\begin{pmatrix}
\sin y + y \cos y - x \sin y
\\
-x \cos y - y \sin y- \cos y
\end{pmatrix} \,,
\\
p_{\rm ex}(x,y) &=  \sin( \pi x) \cos(4 \pi y) \,,
\\
\theta_{\rm ex}(x,y) &=  \sin(4 \pi x) \sin(\pi y)
\,.
\end{aligned}
\]
Notice that the viscosity $\nu$ satisfies Property \textbf{(P0)} whenever $\vert \theta \vert \leq \theta_{\rm max}$.
The domain $\Omega$ is partitioned with the following sequences of polygonal meshes:
\texttt{QUADRILATERAL} distorted meshes,
\texttt{TRIANGULAR} meshes,
\texttt{CVT} (Centroidal Voronoi Tessellations) meshes,
\texttt{RANDOM} Voronoi meshes (see Fig. \ref{fig:meshes}).
For the generation of the Voronoi meshes we used the code \texttt{Polymesher} \cite{TPPM12}.
For each family of meshes we take the sequence with diameter $h =2^{-2}$, $2^{-3}$, $2^{-4}$, $2^{-5}$, $2^{-6}$.

\begin{figure}
\centering
\begin{overpic}[scale=0.18]{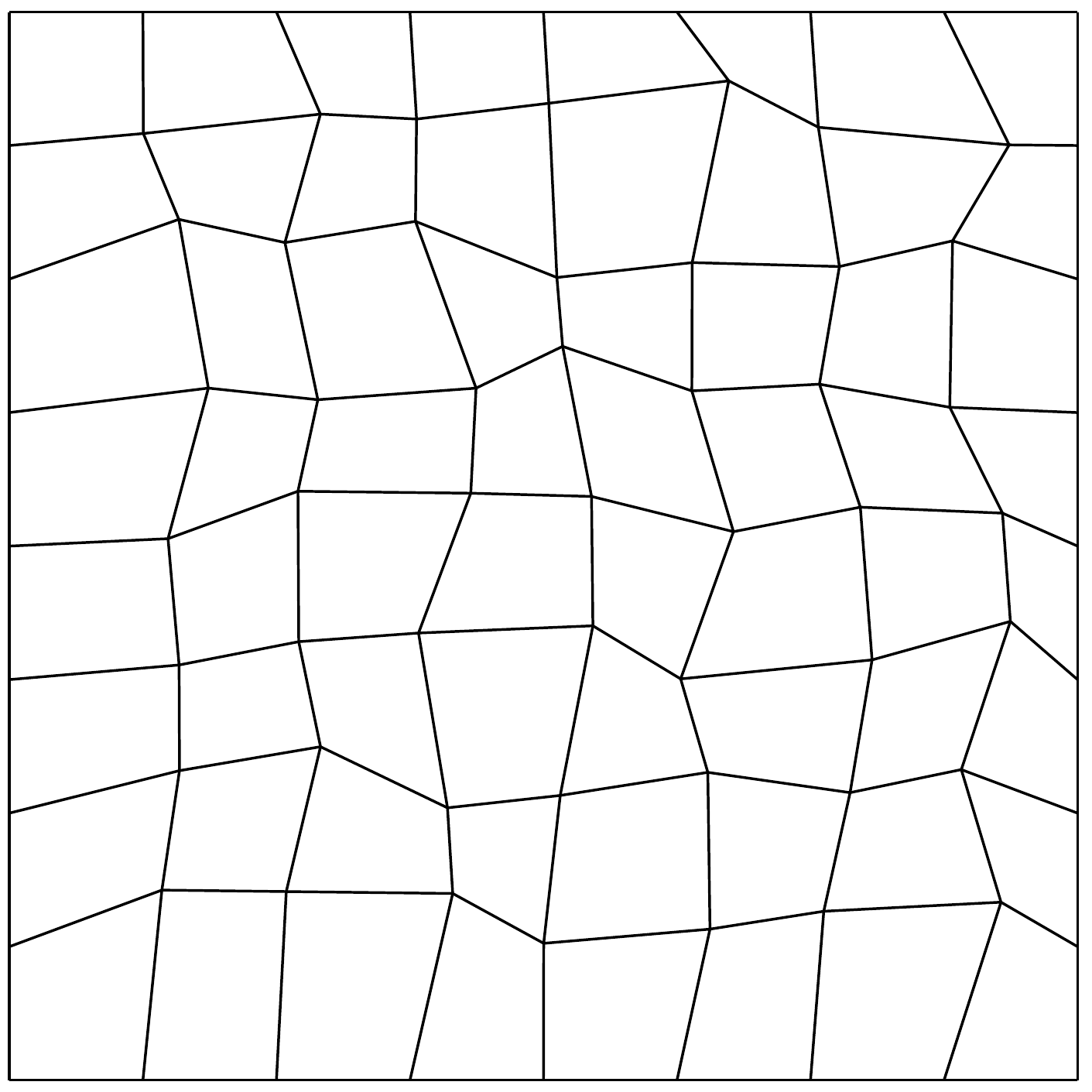}
\put(5,-15){{{\texttt{QUADRILATERAL}}}}
\end{overpic}
\,\,\,
\begin{overpic}[scale=0.18]{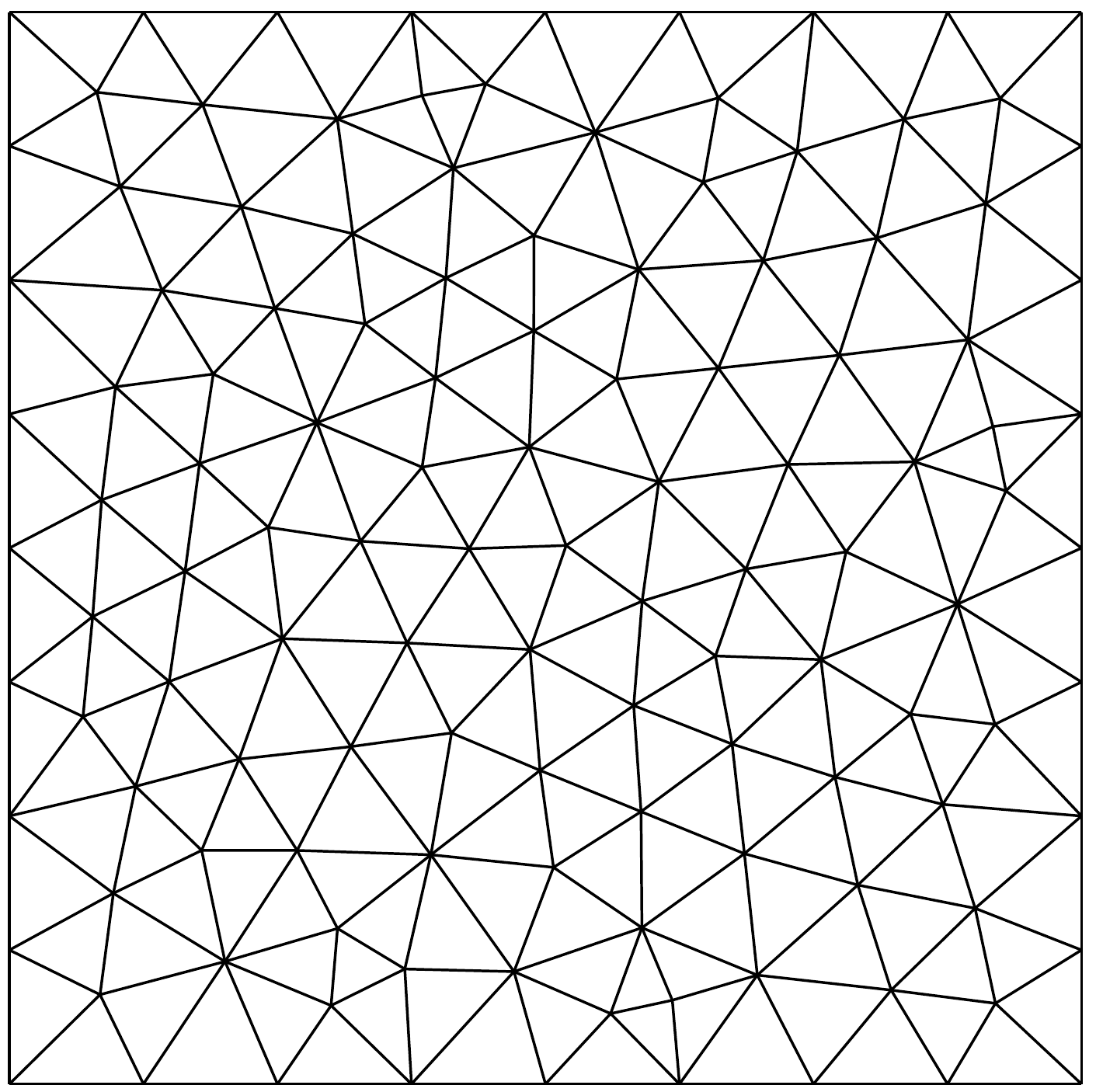}
\put(15,-15){{{\texttt{TRIANGULAR}}}}
\end{overpic}
\,\,\,
\begin{overpic}[scale=0.18]{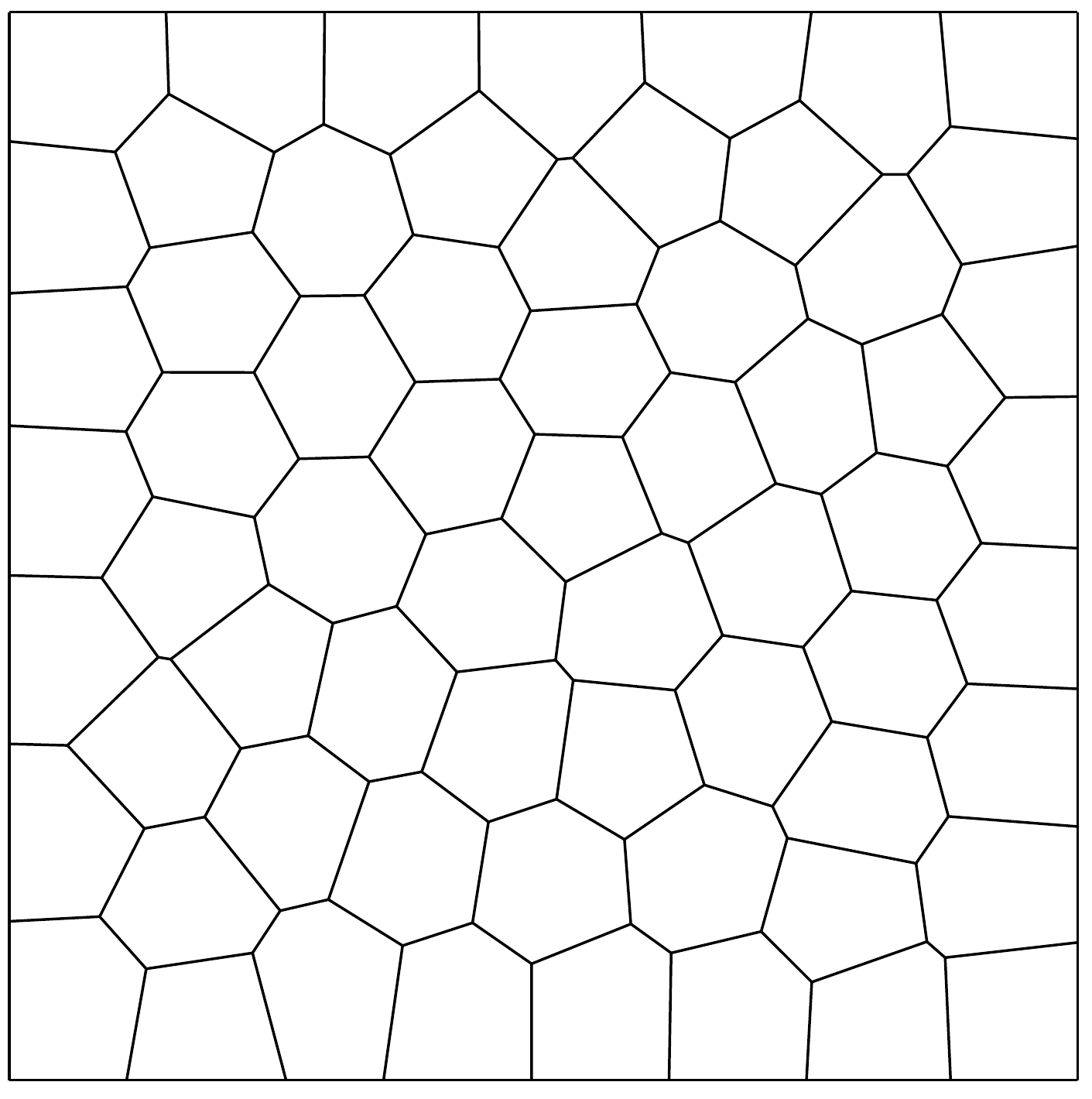}
\put(25,-15){{{\texttt{VORONOI}}}}
\end{overpic}
\,\,\,
\begin{overpic}[scale=0.18]{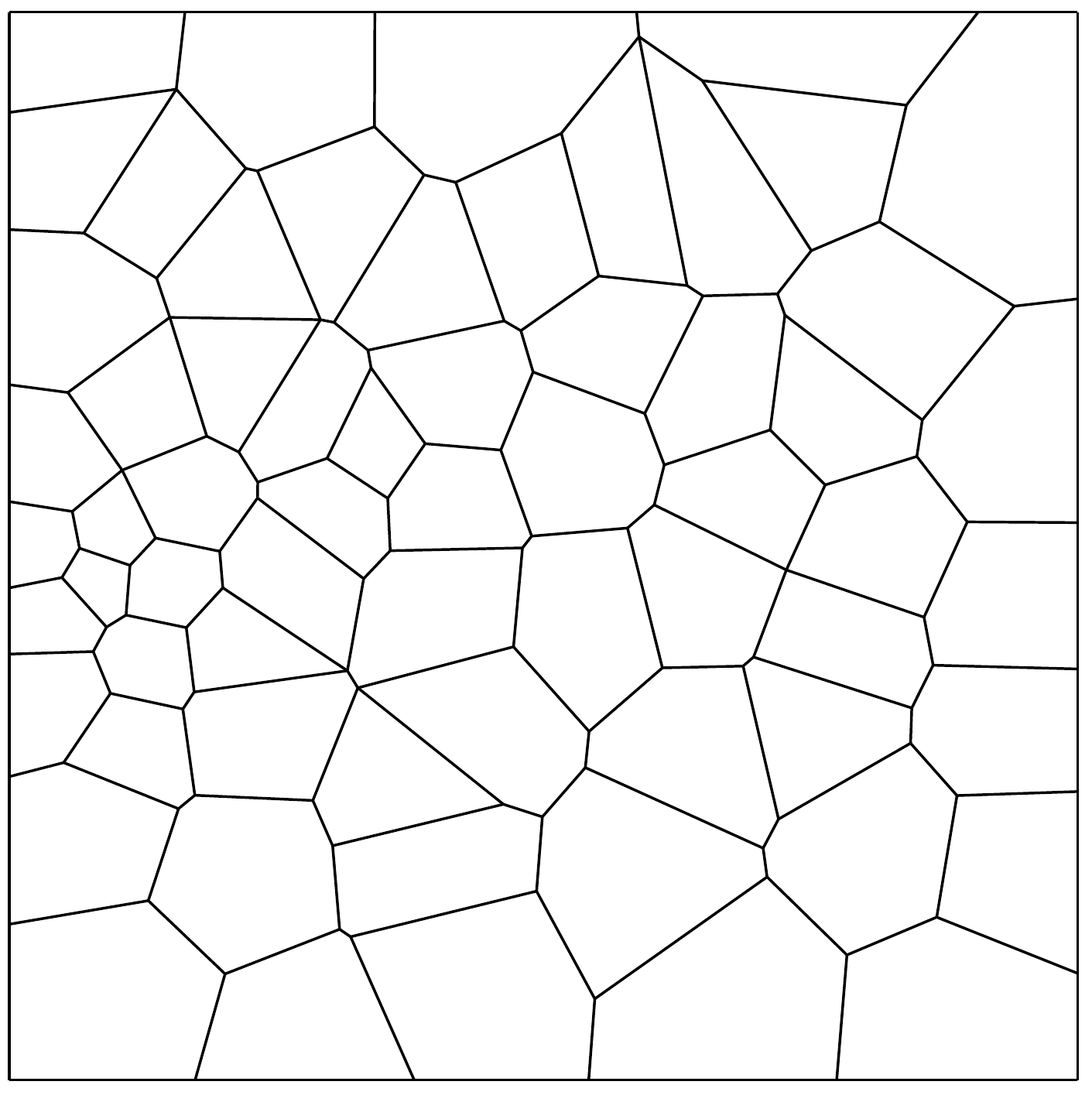}
\put(25,-15){{{\texttt{RANDOM}}}}
\end{overpic}
\vspace{0.5cm}
\caption{Test 1. Example of the adopted polygonal meshes.}
\label{fig:meshes}
\end{figure}

In Table \ref{tab:test1} we show the dimensions of the linear systems solved at each iteration of the fixed-point procedure and number of iterations needed for the adopter sequences of meshes. We iterate the fixed-point procedure with the tolerance \texttt{TOLL=1e-7}.

\begin{table}
\parbox{.45\linewidth}{
\centering
\begin{tabular}{l*{3}{c}}
\toprule
\multicolumn{4}{c}{\texttt{QUADRILATERAL MESHES}} 
\\
\midrule             
   $\texttt{1/h}$  
&  $\texttt{HEAT}$   
&  $\texttt{OSEEN}$       
&  $\texttt{N\_IT}$    
\\
\midrule                          
$\texttt{4}$       
&$\texttt{49}$         &$\texttt{147}$         & $\texttt{6}$          
\\
$\texttt{8}$     
&$\texttt{225}$         &$\texttt{643}$         & $\texttt{6}$        
\\
$\texttt{16}$          
&$\texttt{961}$         &$\texttt{2691}$         & $\texttt{6}$    
\\
$\texttt{32}$       
&$\texttt{3969}$         &$\texttt{11011}$         & $\texttt{5}$     
\\
$\texttt{64}$       
&$\texttt{16129}$         &$\texttt{44547}$         & $\texttt{6}$
\\
\bottomrule
\end{tabular}
}
\qquad \qquad \qquad
\parbox{.45\linewidth}{
\centering
\begin{tabular}{l*{3}{c}}
\toprule
\multicolumn{4}{c}{\texttt{TRIANGULAR MESHES}}
\\
\midrule             
   $\texttt{1/h}$    
&  $\texttt{HEAT}$   
&  $\texttt{OSEEN}$       
&  $\texttt{N\_IT}$     
\\
\midrule                          
$\texttt{4}$       
&$\texttt{129}$         &$\texttt{403}$         & $\texttt{6}$          
\\
$\texttt{8}$     
&$\texttt{605}$         &$\texttt{1847}$         & $\texttt{6}$        
\\
$\texttt{16}$          
&$\texttt{2547}$         &$\texttt{7705}$         & $\texttt{5}$     
\\
$\texttt{32}$       
&$\texttt{10331}$         &$\texttt{31121}$         & $\texttt{5}$      
\\
$\texttt{64}$       
&$\texttt{41985}$         &$\texttt{126211}$         & $\texttt{6}$
\\
\bottomrule
\end{tabular}
}
\\
\vspace{0.5cm}
\parbox{.45\linewidth}{
\centering
\begin{tabular}{l*{3}{c}}
\toprule
\multicolumn{4}{c}{\texttt{VORONOI MESHES}} 
\\
\midrule             
   $\texttt{1/h}$  
&  $\texttt{HEAT}$   
&  $\texttt{OSEEN}$       
&  $\texttt{N\_IT}$    
\\
\midrule                          
$\texttt{4}$       
&$\texttt{63}$         &$\texttt{175}$         & $\texttt{6}$
\\
$\texttt{8}$     
&$\texttt{315}$         &$\texttt{823}$         & $\texttt{6}$        
\\
$\texttt{16}$          
&$\texttt{1399}$         &$\texttt{3567}$         & $\texttt{6}$    
\\
$\texttt{32}$       
&$\texttt{5847}$         &$\texttt{14767}$         & $\texttt{5}$     
\\
$\texttt{64}$       
&$\texttt{23877}$         &$\texttt{60043}$         & $\texttt{6}$
\\
\bottomrule
\end{tabular}
}
\qquad \qquad \qquad
\parbox{.45\linewidth}{
\centering
\begin{tabular}{l*{3}{c}}
\toprule
\multicolumn{4}{c}{\texttt{RANDOM MESHES}}
\\
\midrule             
   $\texttt{1/h}$    
&  $\texttt{HEAT}$   
&  $\texttt{OSEEN}$         
&  $\texttt{N\_IT}$     
\\
\midrule                          
$\texttt{4}$       
&$\texttt{69}$         &$\texttt{187}$         & $\texttt{6}$          
\\
$\texttt{8}$     
&$\texttt{307}$         &$\texttt{807}$         & $\texttt{6}$        
\\
$\texttt{16}$          
&$\texttt{1367}$         &$\texttt{3503}$         & $\texttt{6}$     
\\
$\texttt{32}$       
&$\texttt{5659}$         &$\texttt{14391}$         & $\texttt{5}$      
\\
$\texttt{64}$       
&$\texttt{22921}$         &$\texttt{58131}$         & $\texttt{6}$
\\
\bottomrule
\end{tabular}
}
\vspace{0.25cm}
\caption{Test 1.  Dimensions of the linear systems solved at each iteration of the fixed-point procedure and number of iterations.}
\label{tab:test1}
\end{table} 

In Fig. \ref{fig:results} we display the errors $\texttt{err}(\uu_h, H^1)$, $\texttt{err}(\uu_h, L^2)$, $\texttt{err}(\theta_h, H^1)$, $\texttt{err}(\theta_h, L^2)$ and $\texttt{err}(p_h, L^2)$ for the the sequences of meshes aforementioned. 
We notice that the theoretical predictions of Proposition \ref{prp:conv} are confirmed, moreover we  observe that the method is robust with respect to the mesh distortion.

\begin{figure}
\parbox{.45\linewidth}{
\centering
\begin{overpic}[scale=0.16]{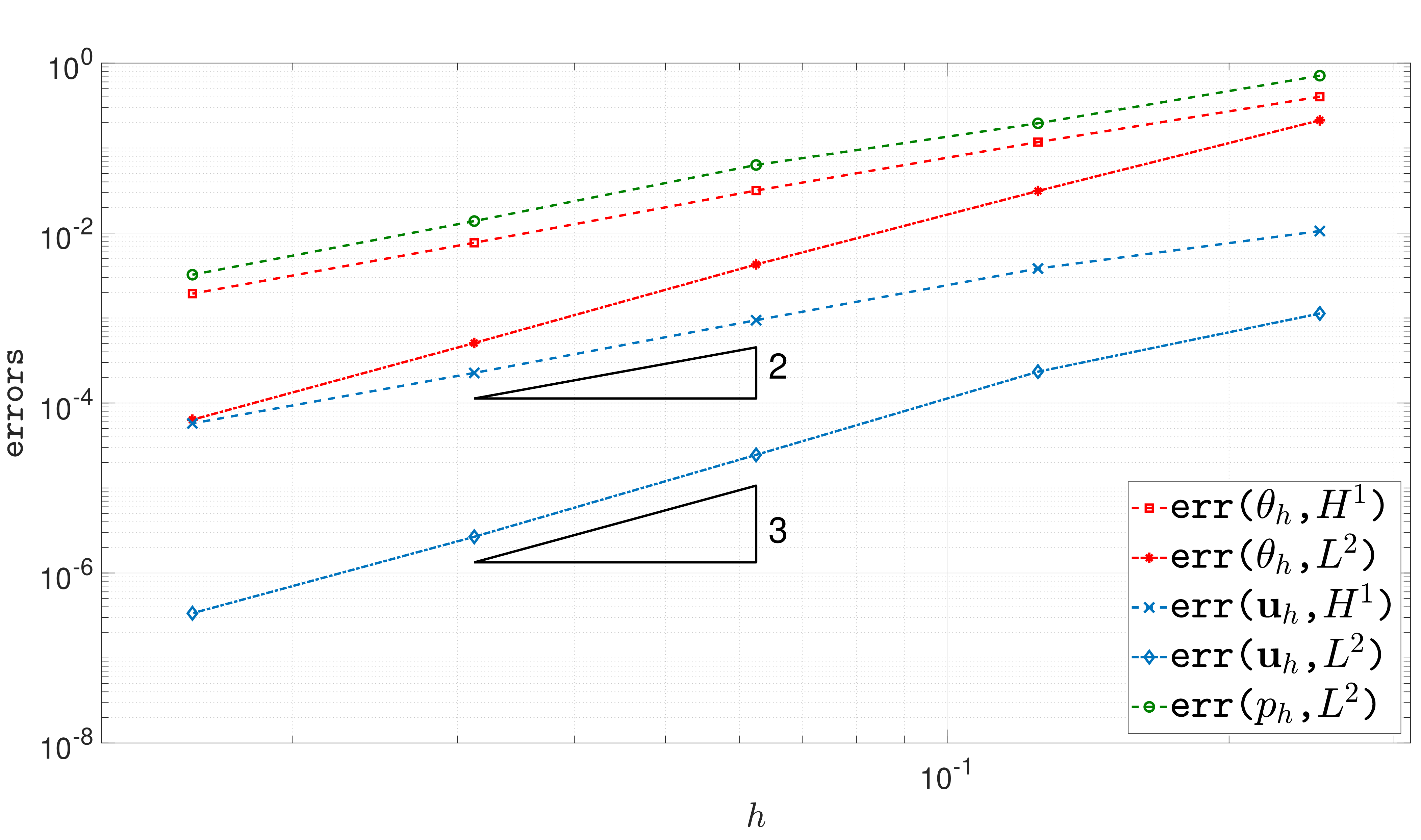}
\put(23,-5){{\texttt{QUADRIALTERAL MESHES}}}
\end{overpic}
}
\qquad
\parbox{.45\linewidth}{
\centering
\begin{overpic}[scale=0.16]{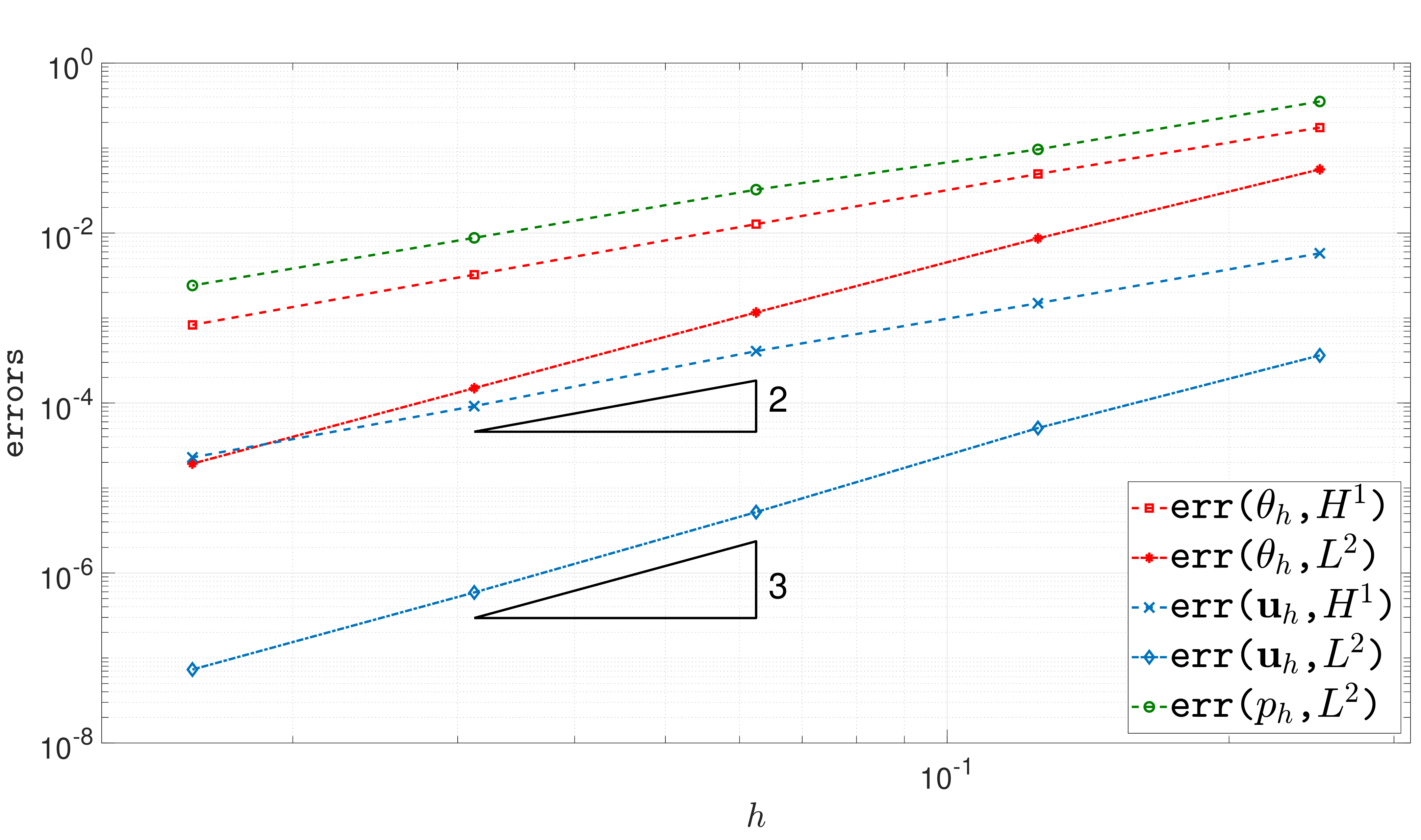}
\put(27,-5){\texttt{TRIANGULAR MESHES}}
\end{overpic}
}
\end{figure}

\begin{figure}
\parbox{.45\linewidth}{
\centering
\begin{overpic}[scale=0.16]{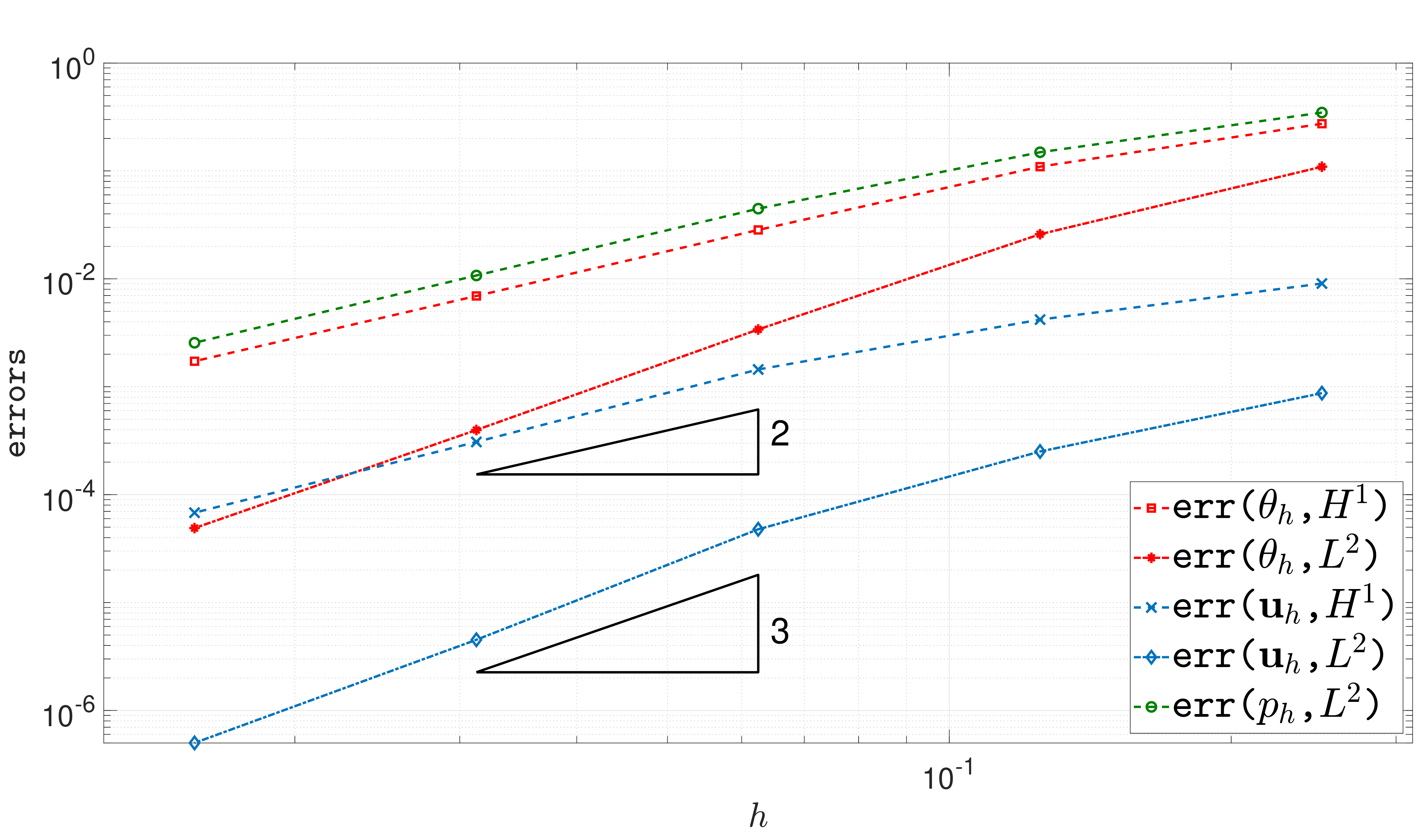}
\put(33,-5){\texttt{VORONOI MESHES}}
\end{overpic}
}
\qquad
\parbox{.45\linewidth}{
\centering
\begin{overpic}[scale=0.16]{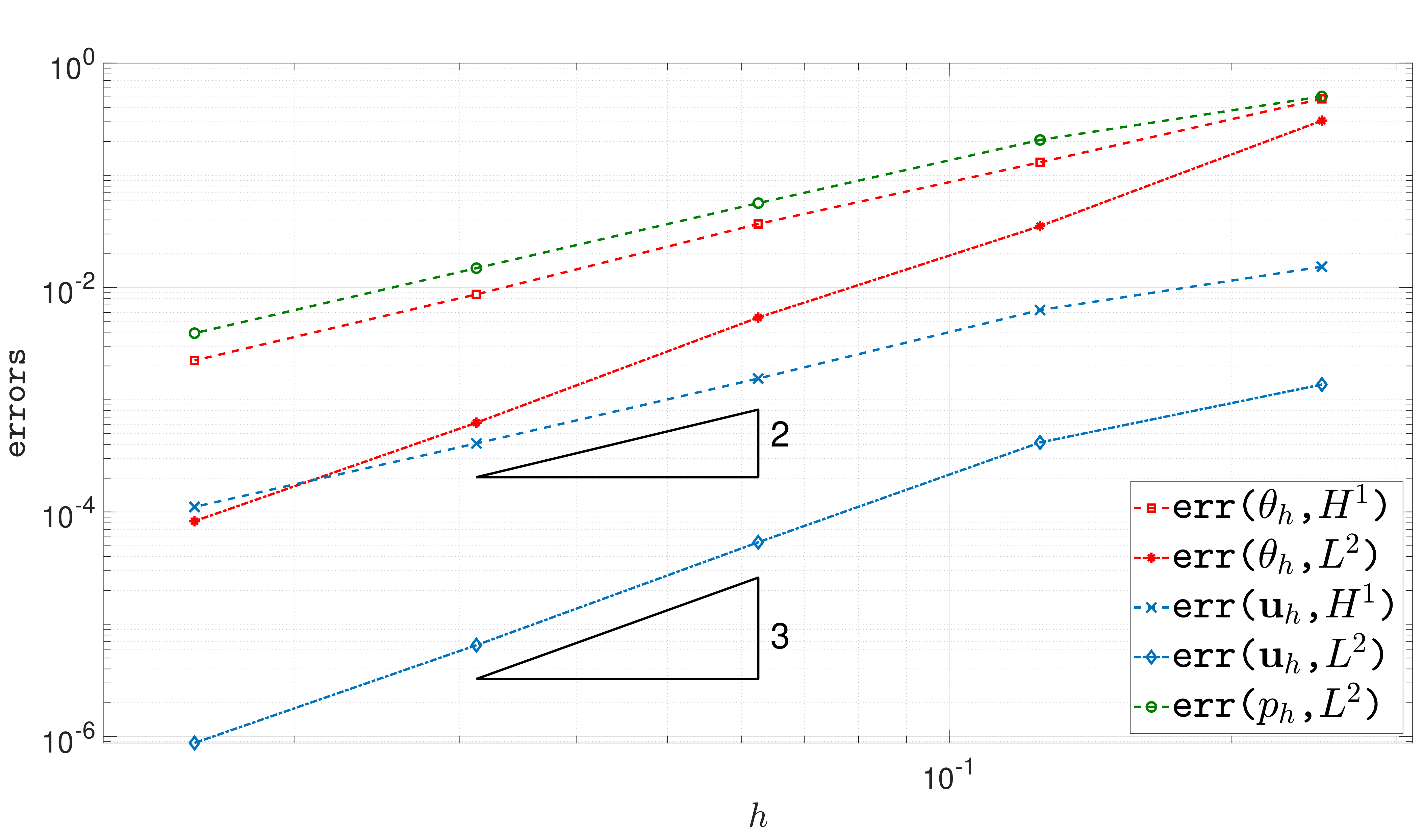}
\put(35,-5){\texttt{RANDOM MESHES}}
\end{overpic}
}
\vspace{1cm}
\caption{Test 1. Convergence histories of the VEM \eqref{eq:pb vem} for the adopted families of polygonal meshes.}
\label{fig:results}
\end{figure}

\paragraph{Test 2. Convection-dominated transport of a passive scalar.}
The scope of the present test is to show that divergence-free discrete velocity solutions might also be of advantage
in coupled problems.
We investigate the behavior of the proposed method for the test  proposed in \cite[Example 6.6]{john-linke-merdon-neilan-rebholz:2017}.
We consider the following coupled problem on the domain $\Omega = (0, 4) \times (0,2) \setminus [2, 4] \times [0, 1]$ with a flow field which is governed by the Stokes equation 
\begin{equation}
\label{eq:pb test}
\left\{
\begin{aligned}
&
\begin{aligned}
-  \nu \, \dd  (\epseps (\uu))  -  \nabla p &= \boldsymbol{0}\qquad  & &\text{in $\Omega$,} \\
\dd \, \uu &= 0 \qquad & &\text{in $\Omega$,} \\
-  \kappa \, \Delta \theta + \uu \cdot \nabla \theta &=  0 \qquad  & &\text{in $\Omega$,} 
\end{aligned}
\end{aligned}
\right.
\end{equation}
coupled with the following boundary conditions
(see Fig. \ref{fig:test2problem} left)
\begin{equation}
\label{eq:bou test}
\left\{
\begin{aligned}
&
\begin{aligned}
\uu &= (0.5\, y(2-y), \, 0)^{\rm T}
  & &\text{on $\Gamma_{\rm in}$,} \\
\uu &= (4\, (y-1)(2-y), \, 0)^{\rm T}
  & &\text{on $\Gamma_{\rm out}$,} \\
\uu &= \boldsymbol{0}
  & &\text{on $\partial \Omega \setminus( \Gamma_{\rm in} \cup \Gamma_{\rm out})$,}
\end{aligned}
\end{aligned}
\right.
\quad
\left\{
\begin{aligned}
&
\begin{aligned}
\theta &= 1
  & &\text{on $\Gamma_{\rm in}$,} \\
\kappa \nabla \theta \cdot \nn &= 0
  & &\text{on $\partial \Omega \setminus \Gamma_{\rm in}$,}
\end{aligned}
\end{aligned}
\right.
\end{equation}
the viscosity is $\nu = \texttt{1e-2}$, the conductivity is $\kappa = \texttt{1e-6}$.
\begin{figure}
\centering
\qquad
\begin{overpic}[scale=0.225]{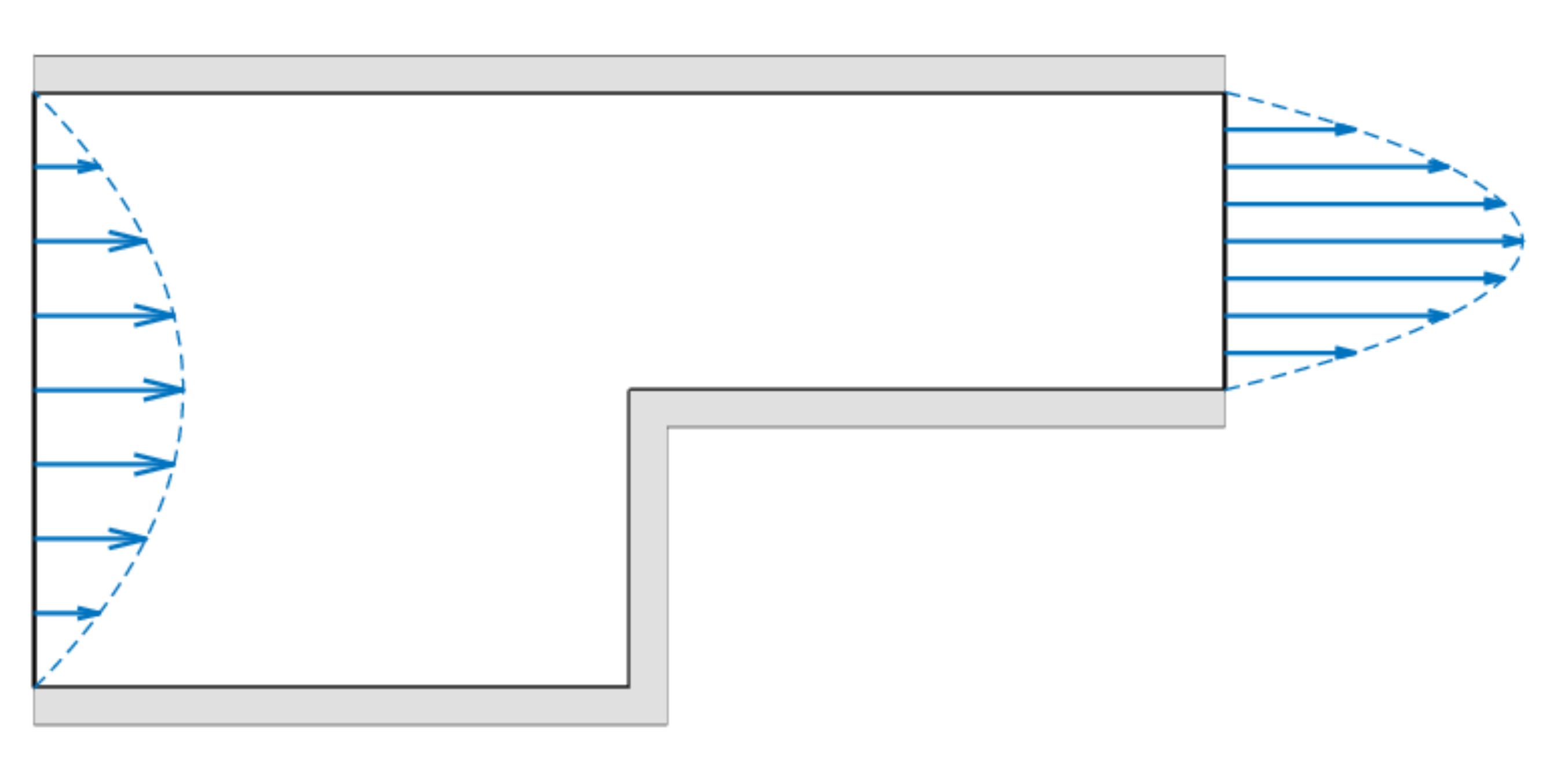}
\put(-7,25){{$\Gamma_{\rm in}$}}
\put(67,32){{$\Gamma_{\rm out}$}}
\end{overpic}
\quad
\begin{overpic}[scale=0.225]{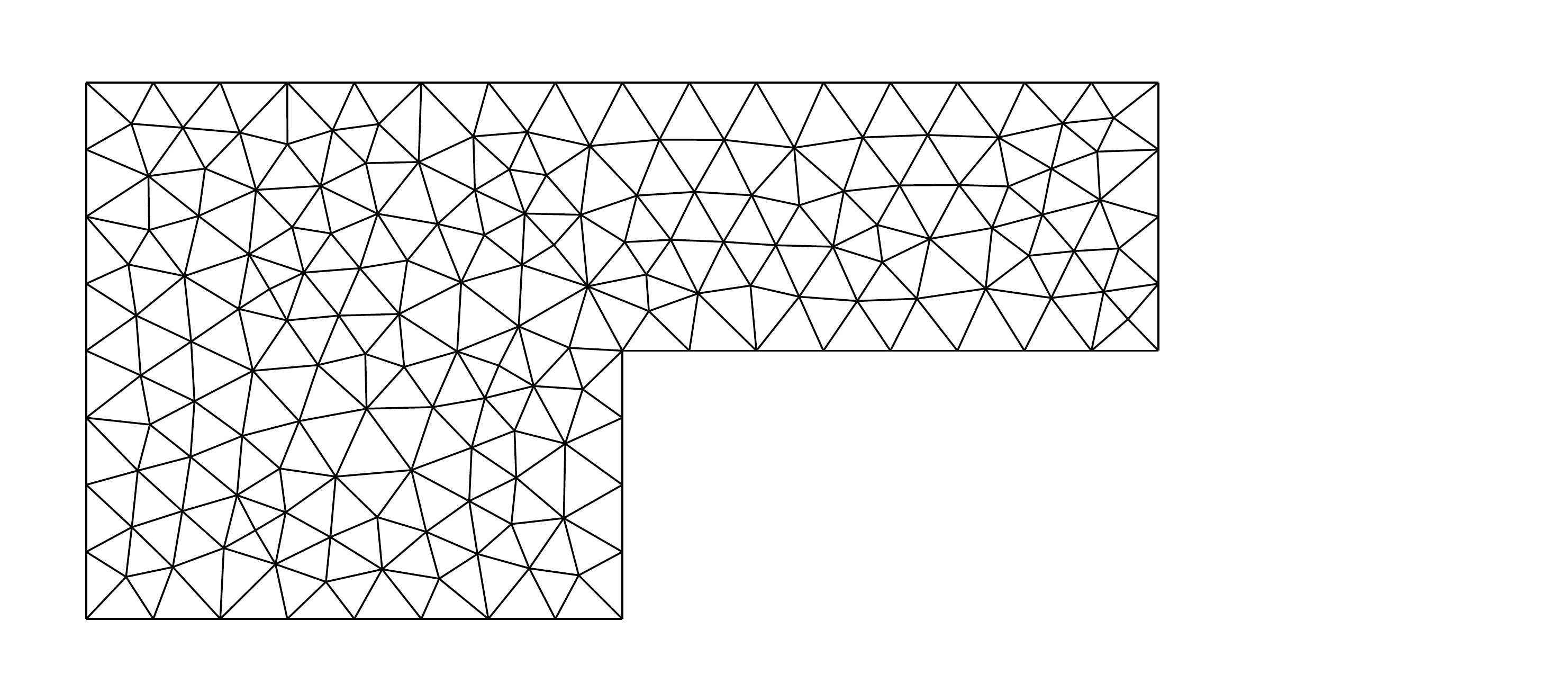}
\end{overpic}
\vspace{0.5cm}
\caption{Test 2. Domain and boundary conditions description (left) and mesh with diameter $h=2^{-2}$ (right).}
\label{fig:test2problem}
\end{figure}
The exact temperature solution is $\theta_{\rm ex} = 1$.

For the sake of the comparison with the results reported in \cite[Example 6.6]{john-linke-merdon-neilan-rebholz:2017}, the domain $\Omega$ is partitioned with a sequence of triangular meshes with diameter $h =2^{-1}$, $2^{-2}$, $2^{-3}$, $2^{-4}$, $2^{-5}$ (see Fig. \ref{fig:test2problem} right).
The continuous convective form associated with the spaces satisfying the boundary conditions \eqref{eq:bou test} is not skew-symmetric, therefore in the discrete scheme we consider the convective form \eqref{eq:forma cshl} in the place of \eqref{eq:forma cshlskew}. 
Note that the heat equation  for the temperature in \eqref{eq:pb test} is convection-dominated. However,  it is worth stressing that, differently form \cite{john-linke-merdon-neilan-rebholz:2017}, we do not employ any stabilization technique to stabilize the discrete problem (we refer to \cite{berrone-supg,BdV-supg,BdV-oseen} for the analysis of the stabilized method in the VEM context both for the convection-dominated elliptic equations and Oseen equations).
Fig. \ref{fig:quiver}, Fig. \ref{fig:test2res} and Table \ref{tab:test2} present the results of the numerical simulations.
As exhibited in \cite[Example 6.6]{john-linke-merdon-neilan-rebholz:2017} the violation of the divergence constraint, as it appears in the classical mixed finite element, pollutes the discrete temperature solution with strong spurious oscillations.
By contrast, the proposed VEM scheme with divergence-free velocity solution computes the temperature with high accuracy as depicted in Fig. \ref{fig:test2res} and Table \ref{tab:test2}.

\begin{figure}
\centering
\begin{overpic}[scale=0.27]{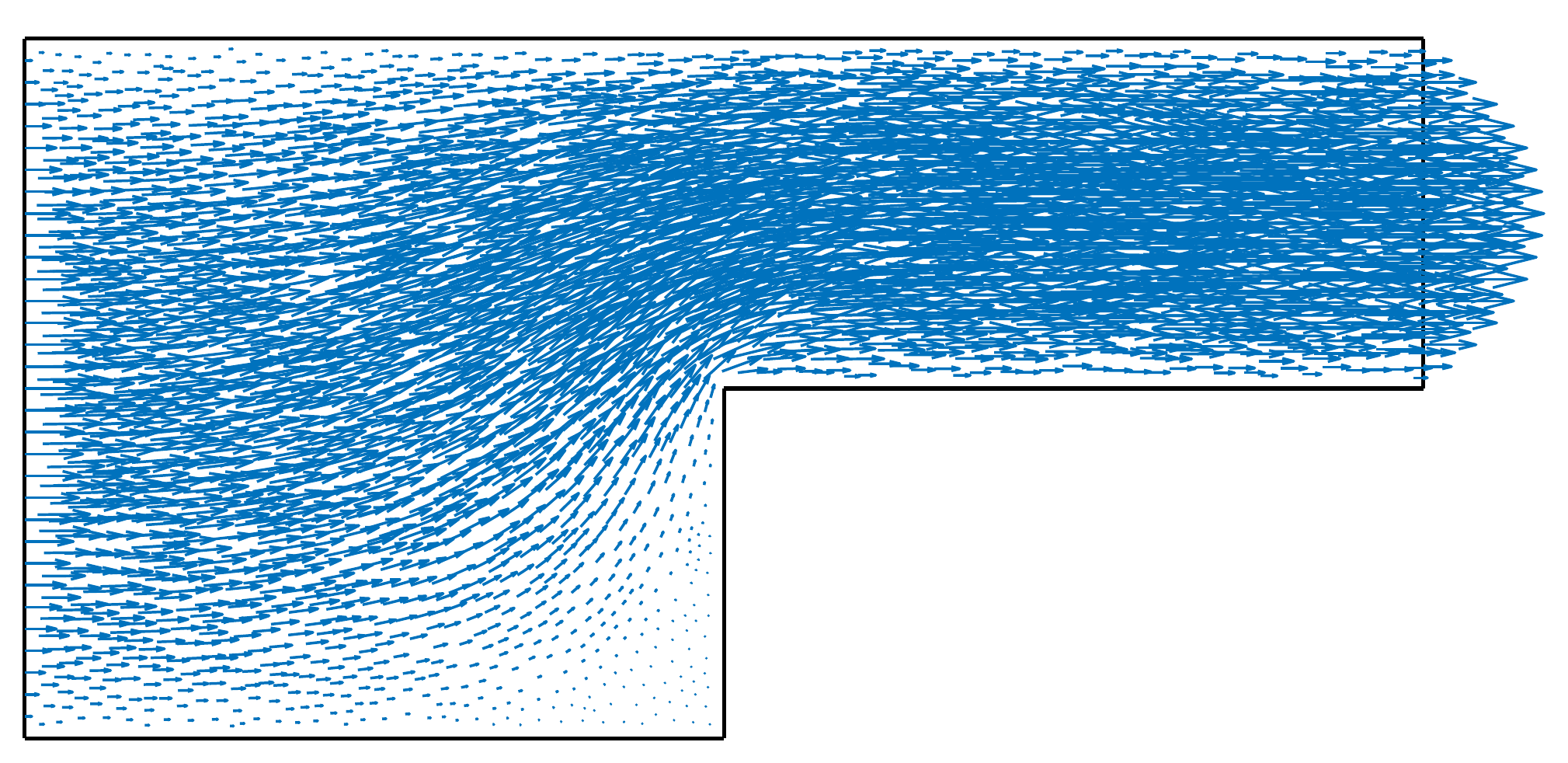}
\end{overpic}
\vspace{0.5cm}
\caption{Test 2. Quiver of the numerical velocity $\uu_h$ with $h=2^{-4}$.}
\label{fig:quiver}
\end{figure}

\begin{figure}
\centering
\begin{overpic}[scale=0.25]{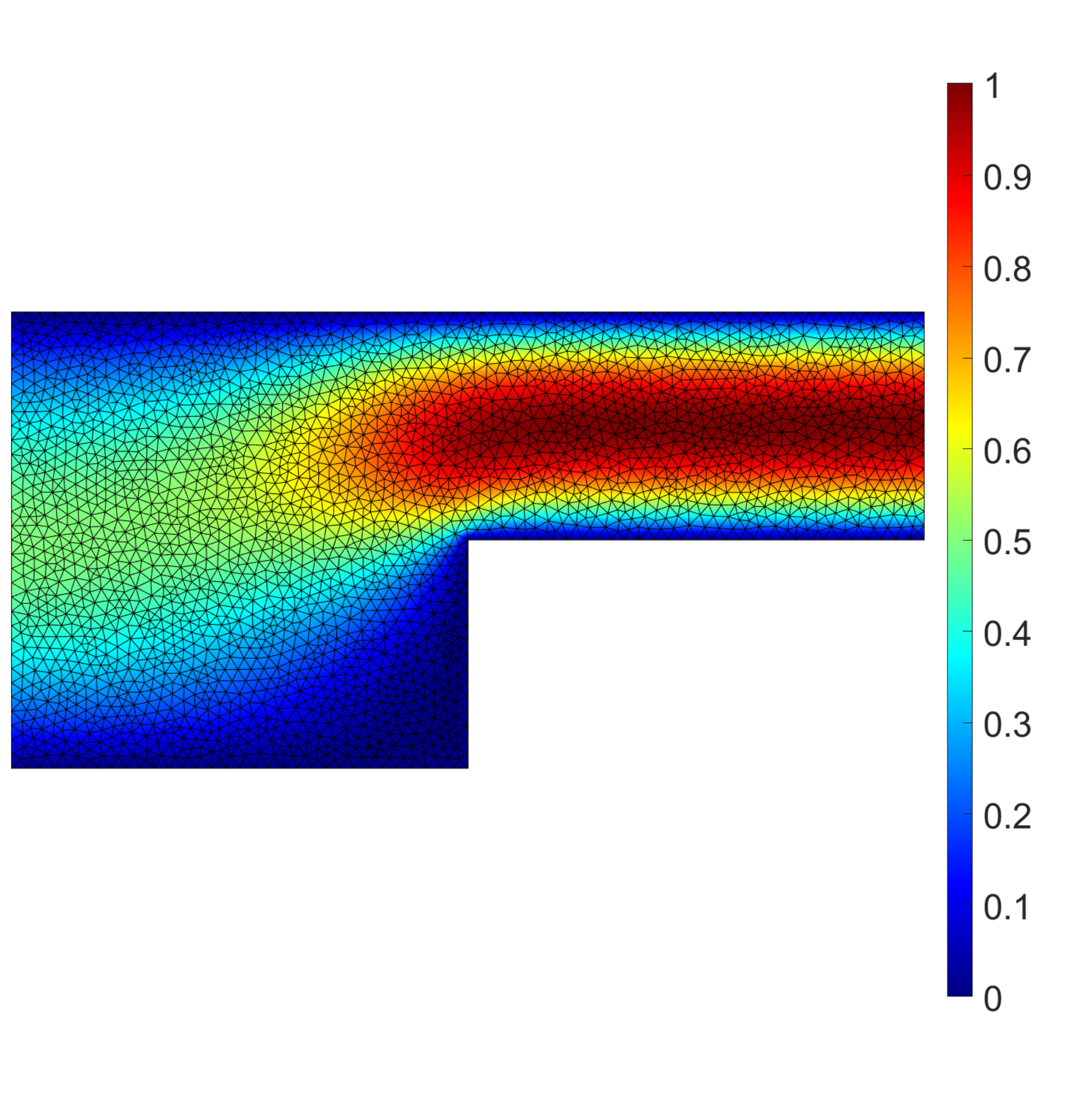}
\end{overpic}
\qquad
\begin{overpic}[scale=0.25]{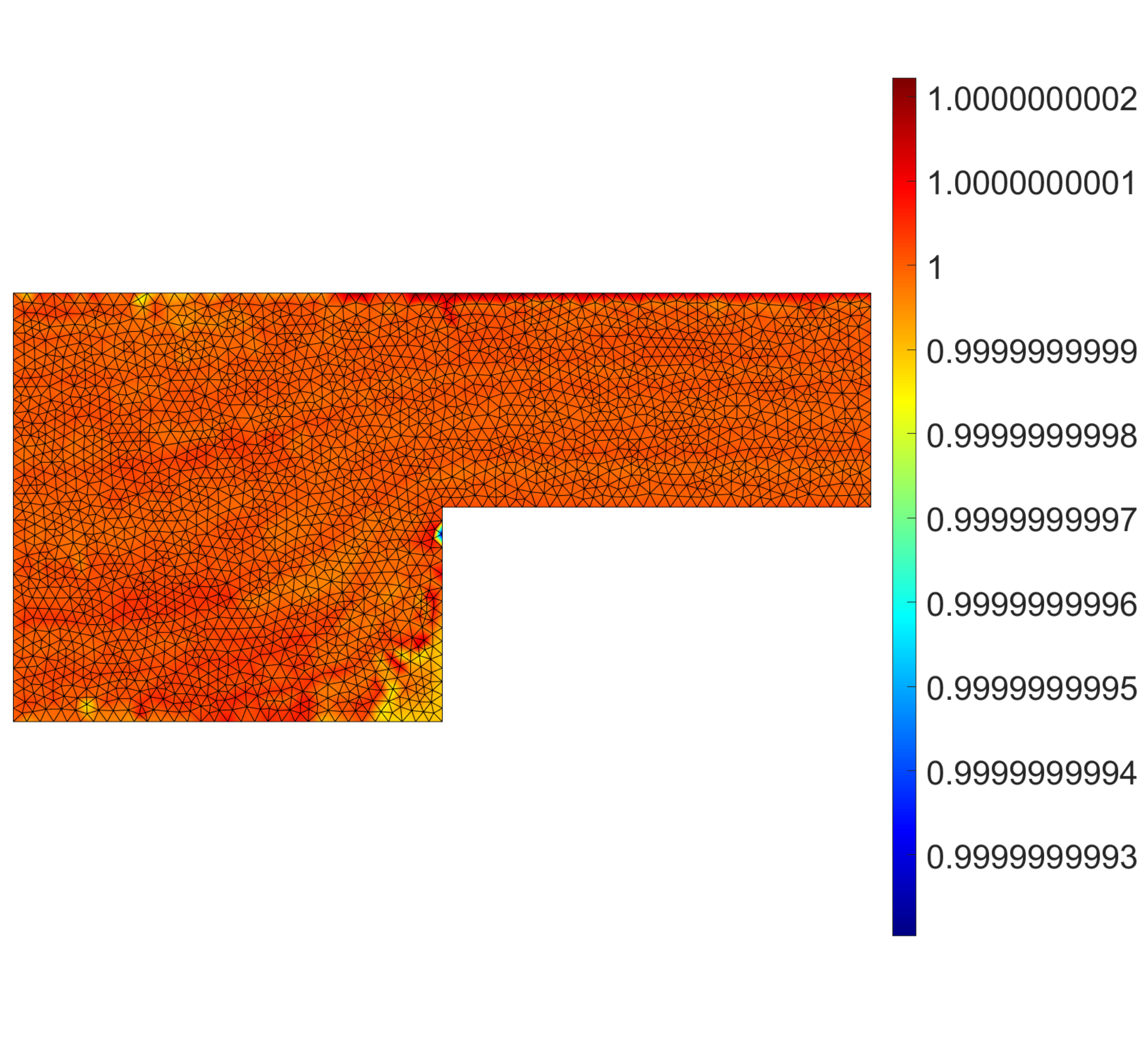}
\end{overpic}
\vspace{-0.25cm}
\caption{Test 2. Absolute value of the velocity $\vert \uu_h\vert$ (left) and temperature $\theta_h$ (right) with $h=2^{-4}$.}
\label{fig:test2res}
\end{figure}

\begin{table}[!h]
\centering
\begin{tabular}{cccccc}
\toprule
$\texttt{1/h}$  & 
$\texttt{2}$    &
$\texttt{4}$    &  $\texttt{8}$ &
$\texttt{16}$   &  $\texttt{32}$ 
\\
\midrule
$\texttt{min}(\theta_h-1)$ &
$\texttt{-7.88e-11}$  &  
$\texttt{-2.78e-10}$ &
$\texttt{-2.87e-08}$  &  
$\texttt{-4.47e-09}$ &   
$\texttt{-3.23e-09}$
\\    
\midrule
$\texttt{max}(\theta_h-1)$ &
$\texttt{+4.50e-11}$  &  
$\texttt{+7.51e-11}$ &
$\texttt{+2.42e-08}$  &  
$\texttt{+2.90e-09}$ &   
$\texttt{+5.34e-09}$
\\    
\bottomrule
\end{tabular}
\vspace{0.25cm}
\caption{Test 2. Minimal and maximal of $\theta - 1$ (\cfun{DoFs} values).}
\label{tab:test2}
\end{table}

\section{Conclusions}
\label{sec:conclusions} 
We presented a Virtual Element discretization of a thermo-fluid dynamic problem modeling the stationary flow of a viscous incompressible fluid (governed by a Navier-Stokes equation),  where the viscosity of the fluid depends on the temperature (governed by an elliptic equation).
We have show the the discrete problem is well-posed and prove optimal error estimates. Numerical experiments which confirm the theoretical error bounds are also presented. Further developments may include the introduction of uncertainty in the thermal conductivity $\kappa$, based on the ideas of \cite{Bonizzoni_Nobile_2014,Bonizzoni_Nobile_2020} as well as the three-dimensional extension, based on employing the three-dimensional VEM  for Navier-Stokes proposed in \cite{Beirao_Dassi_Vacca:2020}.


\section*{Funding}
P.F.A. and M.V. have been partially funded by the research grants PRIN2017 n. \verb+201744KLJL+\emph{``Virtual Element Methods: Analysis and Applications''} and  PRIN2020 n. \verb+20204LN5N5+\emph{``Advanced polyhedral discretisations of heterogeneous PDEs for multiphysics problems''}, funded by the Italian Ministry of Universities and Research (MUR).
P.F.A, G.V. and M.V. are members of INdAM-GNCS.
\bibliographystyle{plain}
\bibliography{references}

\end{document}